\documentclass[11pt]{imsart}

\makeatletter

\@addtoreset{equation}{section}
\makeatletter

\usepackage{amsmath,amssymb,amsthm}
\usepackage{natbib}
\usepackage{threeparttable}
\usepackage{lscape}
\usepackage{bm}
\usepackage{tabularx}
\usepackage{longtable}

\theoremstyle{plain}
\newtheorem{theorem}{Theorem}[section]
\newtheorem{lemma}{Lemma}[section]
\newtheorem{proposition}{Proposition}[section]
\newtheorem{corollary}{Corollary}[section]

\theoremstyle{definition}
\newtheorem{remark}{Remark}[section]
\newtheorem{add}{Addendum}[section]

\DeclareMathOperator{\diag}{diag}

\DeclareMathOperator{\const}{const.}

\begin{document}

\begin{frontmatter}

\title{Two-step estimation of high dimensional additive models}
\runtitle{Two-step estimation}

\begin{aug}
  \author{Kengo Kato}
  \runauthor{K. Kato}
  \affiliation{Hiroshima University}
  \address{1-3-1 Kagamiyama, Higashi-Hiroshima, Hiroshima 739-8526, Japan. E-mail: kkato@hiroshima-u.ac.jp.}
\end{aug}

\begin{abstract}
This paper investigates  the two-step estimation of a high dimensional additive regression model, in which the number of nonparametric additive components is potentially larger than the sample size but the number of significant additive components is sufficiently small. 
The approach investigated consists of two steps. The first step implements the variable selection, typically by the group Lasso, and the second step applies the penalized least squares estimation with Sobolev penalties to the selected additive components.
Such a procedure is computationally simple to implement and, in our numerical experiments, works reasonably well. 
Despite its intuitive nature, the theoretical properties of this two-step procedure have to be carefully analyzed, since the effect of the first step variable selection is random, and generally it may contain redundant additive components and at the same time miss significant additive components. 
This paper derives a generic performance bound on the two-step estimation procedure allowing for these situations, and studies in detail 
the overall performance when the first step variable selection is implemented by the group Lasso. 

\end{abstract}



\end{frontmatter}

\section{Introduction}

In this paper, we are interested in estimating the nonparametric additive regression model
\begin{equation}
y_{i} = c^{*} + g^{*}(\bm{z}_{i}) + u_{i}, \ g^{*}(\bm{z}) = g^{*}_{1}(z_{1}) + \cdots + g^{*}_{d}(z_{d}), \ \mathrm{E}[u_{i} | \bm{z}_{i} ]  = 0, \label{model}
\end{equation}
where $y_{i}$ is a dependent variable and $\bm{z}_{i}=(z_{i1},\dots,z_{id})'$ is a vector of $d$ explanatory variables. Throughout the paper, we assume that the observations are independent and identically distributed (i.i.d.). We presume the situation in which $d$ is larger than the sample size $n$ (in fact we allow for that $d$ is a non-polynomial order in $n$), but most of $g_{1}^{*},\dots,g_{d}^{*}$ are zero functions. So the model of interest is a high dimensional sparse additive model.
Let $\mathcal{Z}$ denote the support of $\bm{z}_{1}$. Without loosing much generality, we assume that $\mathcal{Z} = [0,1]^{d}$.
The unknown additive components $g^{*}_{1},\dots,g^{*}_{d}$ are known to belong to a given class $\mathcal{G}$ of functions on $[0,1]$. 
Throughout the paper, we consider the case in which $\mathcal{G}$ is a Sobolev class: $\mathcal{G} = W_{2}^{\nu}([0,1])$, where $\nu$ is a positive integer and 
\begin{align*}
W_{2}^{\nu}([0,1]) := \Big \{ g:[0,1] \to \mathbb{R} : \ &\text{$g^{(\nu-1)}$ is absolutely continuous such that} \\
&\int_{0}^{1} g^{(\nu)} (z)^{2} dz < \infty \Big  \}.
\end{align*}
For identification of $g^{*}_{1},\dots,g^{*}_{d}$, we assume that 
\begin{equation*}
\mathrm{E}[g^{*}_{j}(z_{1j})] = 0, \ 1 \leq \forall j \leq d.
\end{equation*}
Let $T^{*} := \{ j \in \{1,\dots,d \}: \mathrm{E}[ g_{j}(z_{1j})^{2}] \neq 0 \}$, the index set of nonzero components, and $s^{*} := | T^{*} |$, the number of nonzero components. It is assumed that $s^{*}$ is smaller than $n$. 

There has been a growing interest in estimation of high dimensional sparse additive models \citep{LZ06,RLLW09,MVB09,HHW10,KY10,RWY10,STS11,FFS11,BV11}. Parallel to parametric regression models,
sparsity of the underlying structure makes it possible to estimate consistently the parameter of interest (in this case, the conditional mean function) even when $d$ is larger than $n$.
Estimation accuracy is not a sole goal. In fact, it may happen that, despite the underlying sparsity structure, an estimator containing  many redundant components has a good estimation accuracy. However, to make a better interpretation, one wishes to have a concise model.  Therefore,  the goal is to obtain an estimator that is (i) appropriately sparse, in the sense that 
it does not contain many redundant additive components, and at the same time (ii) possesses a good estimation accuracy.   
 
A distinctive feature of the present nonparametric estimation, when compared with the parametric case, is that the function class $\mathcal{G}$ is much more complex, which brings a new challenge.
To address this problem, in a fundamental paper,  \cite{MVB09}, they proposed a penalized least squares estimation method with the novel penalty term:
\begin{equation}
\tilde{\lambda}_{1} \sum_{j=1}^{d} \sqrt{ \| g_{j} \|_{2,n}^{2} +\tilde{\lambda}_{2} I(g_{j})^{2} } + \tilde{\lambda}_{3}\sum_{j=1}^{d} I(g_{j})^{2}, \label{mvbpenalty}
\end{equation}
where
\begin{equation*}
\| g_{j} \|_{2,n}^{2} := \frac{1}{n} \sum_{i=1}^{n} g_{j}(z_{ij})^{2}, I(g_{j})^{2} := \int_{0}^{1} g_{j}^{(\nu)}(z_{j})^{2}dz_{j}.
\end{equation*}
The term $\| g_{j} \|_{2,n}$ penalizes the event that $g_{j}$ enters the model, thereby to enforce sparsity of the resulting estimator; the term $I(g_{j})$ penalizes roughness of $g_{j}$ and controls the complexity of the class $\mathcal{G}$, thereby to avoid an overfitting and guarantee a good estimation accuracy of the resulting estimator.
See \cite{KY10,RWY10,STS11} for a further progress. An important theoretical fact is that, as \cite{STS11} showed\footnote{\cite{STS11} adopted a slightly different formulation than \cite{MVB09}, i.e., in \cite{STS11}, the Sobolev penalty $I(g_{j})$ is replaced by the reproducing kernel Hilbert space norm. However, essentially the same proof applies to the original \cite{MVB09} estimator.}, under suitable regularity conditions such as the uniform boundedness of the error term, 
the \cite{MVB09} estimator achieves the minimax  rate of convergence (in the $L_{2}$-risk) $s^{*} \delta^{2}$ where
\begin{equation*}
\delta := \max \left \{ n^{-\nu/(2\nu+1)}, \sqrt{\frac{\log d}{n}} \right \}.
\end{equation*}
See \cite{RWY10} and \cite{STS11} for minimax rates in our problem. 
Thus, from a theoretical point of view, their estimator has a good convergence property. 

However, we would like to point out that the double penalization strategy that 
\cite{MVB09} used may practically lead to a loss of accuracy in estimation/variable selection. In practice, the sparsity penalty brings a shrinkage bias to the selected additive components, so the resulting estimator may have a worse performance than an oracle estimator, which is an ``estimator'' constructed as if $T^{*}$ were known,  even when the correct model selection is achieved. Furthermore, choosing the tuning parameters in such a way that the estimation accuracy is optimized would result in including too many redundant variables. 
The problem of shrinkage bias caused by sparsity penalties has been recognized in the parametric regression case. 
In the linear regression case, \cite{BC09} considered the two-step estimator of the coefficient vector, which corresponds to the least squares estimator applied to the variables selected by the Lasso \citep{Ti96}, and observed that in their simulation experiments the two-step estimator significantly outperforms the Lasso estimator because the former can remove a shrinkage bias. Motivated by these observations, we consider the two-step estimation of high dimensional additive models in which the first step implements the variable selection, typically by the group Lasso \citep{YL06}, and the second step applies the penalized least squares estimation with Sobolev penalties to the selected additive components. The paper is devoted to a careful study of the theoretical and numerical properties of this two-step estimator.

The main theoretical finding of the paper is to derive a generic bound on the $L_{2}$-risk of the second step estimator. In a typical situation, the bound reduces to 
\begin{equation}
\max \left \{ s^{*} n^{-2\nu/(2\nu+1)}, | \hat{T} \backslash T^{*} |\delta^{2}, \| {\textstyle \sum}_{j \in T^{*} \backslash \hat{T}} g_{j}^{*} \|^{2}_{2} \right \}, \label{ibound}
\end{equation}
where $\hat{T} \subset \{ 1,\dots, d \}$ is the index set selected by the first step variable selection. 
Importantly,  this bound applies to any variable selection method such that, roughly speaking, the size $| \hat{T} |$ is stochastically not overly large compared with $s^{*}$, and holds in both the situations in which (i) $\hat{T}$ may have redundant variables (i.e., $\hat{T} \backslash T^{*} \neq \emptyset$), and (ii) $\hat{T}$ may miss significant variables (i.e., $T^{*} \backslash \hat{T} \neq \emptyset$). This bound has a natural  interpretation. 
The first term $s^{*} n^{-2\nu/(2\nu+1)}$ corresponds to the oracle rate, the rate that could be achieved when $T^{*}$ were known; the second term $| \hat{T} \backslash T^{*} | \delta^{2}$ corresponds to the effect of selecting redundant variables; the third term $\| \sum_{j \in T^{*} \backslash \hat{T}} g_{j}^{*} \|^{2}_{2}$
corresponds to the effect of missing significant variables. 

One may wonder that it is plausible  to presume the perfect model selection (i.e., $\hat{T} = T^{*}$ with probability approaching one), in which case the analysis becomes trivial, since one may guarantee the perfect model selection by applying a hard thresholding method to the first step group Lasso or using the adaptive group Lasso. 
However, what we need is a bound that applies to a general situation in which $\hat{T}$ may fail to recover $T^{*}$. In view of the literature, 
to guarantee the perfect model selection requires a side condition that the non-zero additive components are well separated from zero (in the $L_{2}$-sense), which is considerably restrictive from a theoretical point of view. In fact, under the side condition, the exact oracle rate $s^{*} n^{-2\nu/(2\nu+1)}$ will be  achievable, and in view of the minimax rate, this means that the complexity of the problem is significantly reduced.\footnote{$\sqrt{\log d/n}$ can be dominant in $\delta$ as long as $\log d/n^{1/(2\nu+1)} \to \infty$. Our analysis allows for this case.}
Therefore, to make a meaningful comparison with existing estimators such as the \cite{MVB09} estimator, one has to 
establish a performance bound without presuming the perfect model selection. 
An important aspect of the bound (\ref{ibound}) is that it characterizes the effect of the first step variable selection in an explicit manner, which makes the analysis non-trivial.
Another interesting finding is that, despite the random fluctuation of $\hat{T}$, the smoothing penalty level in the second step can be taken independent of $d$.

In this paper, we primarily focus on to use the group Lasso as a first step variable selection method. 
The side (and hence not main) contribution of the paper is to establish (some) refined asymptotic results on the statistical properties of the group Lasso estimator for high dimensional additive models, which complements the recent literature on the theoretical study of the group Lasso \citep{NR08,B08,WH10,HZ10,HHW10,LPTV10,NRWY10}. 
The group Lasso, when applied to estimation of additive models, is based on a different idea of dealing with the complexity of function classes, i.e., 
approximating each function class by a finite dimensional class of functions and controlling the complexity by its dimension.\footnote{See Chapter 10 of \cite{vdG00} for two different ideas, namely, penalty and sieve approaches, to deal with the complexity of function classes in nonparametric regression.} 
Expanding each additive component by a linear combination of given basis functions, selection of additive components reduces to selection of groups of the coefficient vector to 
the basis expansion, so that the group Lasso turns out to be an effective way of selecting additive components. 
Combined with the bound (\ref{ibound}), when the group Lasso is used as a first step procedure, it will be seen that (under suitable regularity conditions, of course)  (i) the two-step estimator is at least as good as the \cite{MVB09} estimator, meaning that it achieves the rate $s^{*} \delta^{2}$ in general cases in which $\hat{T}$ may fail to
recover $T^{*}$ (so $\hat{T} \backslash T^{*} \neq \emptyset$ or $T^{*} \backslash \hat{T} \neq \emptyset$, or both); (ii) if it happens that the perfect model selection holds, then the two-step estimator enjoys the exact oracle rate $s^{*} n^{-2\nu/(2\nu+1)}$; (iii) the second step estimation can automatically adapt to both situations (i) and (ii), i.e., adapt to the model selection ability of $\hat{T}$. 
We believe that these theoretical results in the context of estimation of high dimensional additive models are useful.

We also carry out simulation experiments to investigate the finite sample property of the two-step estimator. The simulation results suggest that the proposed two-step estimator is a good alternative in estimating high dimensional additive models. 

There are a large number of works on the theoretical analysis of penalized estimation methods for high dimensional sparse models, especially on the Lasso for linear regression models \citep{BTW07a,BTW07b,ZY07,ZH08,MY09,W09,CP09,BRT09,Z09}, generalized linear  models \citep{vdG08,NRWY10} and quantile regression models \citep{BC11}. See also \cite{BV11} for a recent review. In the quantile regression context, \cite{BC11} formally established the theoretical properties of the post-penalized estimator that corresponds to the unpenalized quantile regression estimator applied to the variables selected by the $\ell_{1}$-penalized estimator. Their analysis is extended to the linear regression case in \cite{BC09}. As noted before, the present paper builds on these fundamental papers, but 
has two important theoretical departures from the previous analysis: (i) the model of interest is a nonparametric additive model, and (ii) the second step estimation has smoothness penalty terms.

The remainder of the paper is organized as follows.
Section 2 describes the two-step estimation method. 
Section 3 presents some simulation experiments. 
Section 4 is devoted to the theoretical study. 
Section 5 concludes.
Section 6 provides a proof of Theorem \ref{thm1}. 
Some other technical proofs are gathered in Appendices.

{\bf Notation}: In the theoretical study, we rely on the asymptotic scheme in which $d$ and $s^{*}$ may diverge as the sample size $n$. Hence we agree that all parameters values (such as $d, s^{*} \dots$) are index by $n$ and the limit is always taken as $n \to \infty$, but we omit the the index $n$ in most cases. 
For two sequences $a=a(n)$ and $b=b(n)$, we use the notation $a \lesssim b$ 
if there exists a positive constant $C$ independent of $n$ such that $a \leq C b$, $a \asymp b$ if $a \lesssim b$ and $b \lesssim a$, and $a \lesssim_{p} b$ if $a = O_{p}(b)$.
Let $\mathbb{S}^{l-1}$ denote the unit sphere on $\mathbb{R}^{l}$ for a positive integer $l$. 
Let $\bm{I}_{l}$ denote the $l \times l$ identity matrix. We use $\| \cdot \|_{E}$ to indicate the Euclidean norm, and let $\| \cdot \|_{\infty}$ denote the supremum norm.
For a matrix $\bm{A}$, let $\| \bm{A} \|$ denote the operator norm of $\bm{A}$. For a symmetric positive semidefinite matrix $\bm{A}$, let $\bm{A}^{1/2}$ denote the symmetric square root matrix of $\bm{A}$. Let $\| \cdot \|_{2,n}$ and $\| \cdot \|_{2}$ denote the empirical and population $L_{2}$ norms with respect to $\bm{z}_{i}$'s respectively, i.e., 
for $g:\mathcal{Z} \to \mathbb{R}$, 
\begin{equation*}
\| g \|^{2}_{2,n} := \frac{1}{n} \sum_{i=1}^{n} g (\bm{z}_{i})^{2}, \ \| g \|^{2}_{2} := \mathrm{E}[ g(\bm{z}_{1})^{2}].
\end{equation*}
To make the notation simpler, if we write the index $j$ in $g_{j}:[0,1] \to \mathbb{R}$, we agree that 
\begin{equation*}
\| g_{j} \|^{2}_{2,n} := \frac{1}{n} \sum_{i=1}^{n} g_{j} (z_{ij})^{2}, \ \| g_{j} \|^{2}_{2} := \mathrm{E}[ g_{j}(z_{1j})^{2}].
\end{equation*}

\section{Two-step estimation}

This section describes the proposed estimation method. 

{\bf First step}: Use an appropriate variable selection method to determine a subset $\hat{T}$ of $\{1,\dots, d \}$. 

{\bf Second step}: 
Apply a penalized least squares method with roughness penalties to the selected additive components: 
\begin{multline}
(\tilde{c}, \tilde{g}_{j}, j \in \hat{T}) \\
:= \arg \min_{c \in \mathbb{R}, g_{j} \in \mathcal{G}, j \in \hat{T}} \left [ \frac{1}{2n} \sum_{i=1}^{n} (y_{i} - c - \sum_{j \in \hat{T}} g_{j}(z_{ij}))^{2} + \sum_{j \in \hat{T}} \lambda_{2,j}^{2} I(g_{j})^{2} \right ],
\label{pl}
\end{multline}
subject to the restrictions $\sum_{i=1}^{n} g_{j}(z_{ij}) = 0, \ \forall j \in \hat{T}$,
where $\lambda_{2,j} \geq 0$ are smoothing parameter and the term $I(\cdot)$ is the Sobolev penalty: 
\begin{equation*}
I(f)^{2} := \int_{0}^{1} f^{(\nu)} (z)^{2} dz.
\end{equation*}
The resulting estimator of $g^{*}$ is given by $\tilde{g} (\bm{z}) :=\sum_{j \in \hat{T}} \tilde{g}_{j}(z_{j})$. In the theoretical study, to make the argument simple, we let 
\begin{equation*}
\lambda_{2,1} = \cdots = \lambda_{2,d} = \lambda_{2}.
\end{equation*}
It will be shown that $\lambda_{2} \asymp n^{-\nu/(2\nu+1)}$ gives a correct choice. 

There are several possible choices for the first step variable selection method. We primarily focus on to use the group Lasso.

{\bf Group Lasso}: Suppose that we have a set of basis functions $\{ \psi_{1},\dots,\psi_{m} \}$ on $[0,1]$ (except for the constant function).
The number $m = m_{n}$ should be taken such that $m \to \infty$ as $n \to \infty$ but $m = o(n)$. It will be shown that $m \asymp n^{1/(2\nu+1)}$ gives an optimal choice.  
We estimate each additive component by a linear combination of basis functions.
Let $\mathcal{G}_{m} := \{ g: [0,1] \to \mathbb{R} : g(\cdot) = c+\sum_{k=1}^{m} \beta_{k} \psi_{k}(\cdot), c \in \mathbb{R}, \beta_{k} \in \mathbb{R}, 1 \leq k \leq m \}$. We consider the estimator:
\begin{multline}
(\hat{c}, \hat{g}_{1}, \dots, \hat{g}_{d}) \\
:= \arg \min_{c \in \mathbb{R}, g_{j} \in \mathcal{G}_{m}, 1 \leq j \leq d} \left [ \frac{1}{2n} \sum_{i=1}^{n} \{ y_{i} - c-  {\textstyle \sum}_{j=1}^{d} g_{j}(z_{ij}) \}^{2} + \frac{\sqrt{m}\lambda_{1}}{n} \sum_{j=1}^{d} \| g_{j} \|_{2,n} \right ],
\label{glasso2}
\end{multline}
subject to the restrictions $\sum_{i=1}^{n} g_{j}(z_{ij}) = 0, \ 1 \leq \forall j \leq d$,
where $\lambda_{1}$ is a nonnegative tuning parameter that controls sparsity of the resulting estimator. The resulting  estimator of $g^{*}$ is given by $\hat{g}(\bm{z}) := \sum_{j=1}^{d} \hat{g}_{j}(z_{j})$. 
It will be shown that 
\begin{equation*}
\lambda_{1} \asymp \max \left \{ \sqrt{n}, \sqrt{\frac{n\log d}{m}} \right \}
\end{equation*}
gives a correct choice. Let $\hat{T}^{0}:= \{ j \in \{ 1,\dots, d \} : \| \hat{g}_{j} \|_{2,n} > 0 \}$.

It is more convenient to concentrate out the constant term when analyzing the estimator $\hat{g}$. Define $\bm{x}_{iG_{j}} := (\psi_{1}(z_{ij}),\dots,\psi_{m}(z_{ij}))'$, $\tilde{\bm{x}}_{iG_{j}} := \bm{x}_{iG_{j}} - \bar{\bm{x}}_{G_{j}} \ (\bar{\bm{x}}_{G_{j}} := n^{-1} \sum_{i=1}^{n} \bm{x}_{iG_{j}})$ for $1 \leq j \leq d$ and $\tilde{\bm{x}}_{i}:=(\tilde{\bm{x}}_{iG_{1}}',\tilde{\bm{x}}_{iG_{2}}',\dots,\tilde{\bm{x}}_{iG_{d}}')'$. Let $\hat{\bm{\Sigma}}_{j}
 := n^{-1} \sum_{i=1}^{n} \tilde{\bm{x}}_{iG_{j}} \tilde{\bm{x}}_{iG_{j}}'$ for $1 \leq j \leq d$ and $\hat{\bm{\Sigma}} := n^{-1} \sum_{i=1}^{n} \tilde{\bm{x}}_{i} \tilde{\bm{x}}_{i}'$. For $\bm{\beta} = (\beta_{11},\dots,\beta_{1m},\beta_{21},\dots,\beta_{dm})' \in \mathbb{R}^{dm}$, we use the notation $\bm{\beta}_{G_{j}} = (\beta_{j1},\dots,\beta_{jm})'$ for $1 \leq j \leq d$.
Working with this notation, it is seen that $\hat{c} = n^{-1} \sum_{i=1}^{n} (y_{i} - \sum_{j=1}^{d} \hat{g}_{j}(z_{ij})) = \bar{y} := n^{-1} \sum_{i=1}^{n} y_{i}$ and  $\hat{g}_{j}(z_{j}) = \sum_{k=1}^{m} \hat{\beta}_{jk} (\psi_{k}(z_{j}) - \bar{\psi}_{jk}) \ (\bar{\psi}_{jk} := n^{-1} \sum_{i=1}^{n} \psi_{k}(z_{ij}); 1 \leq j \leq d, 1 \leq k \leq m)$, where 
\begin{equation}
\hat{\bm{\beta}} := \arg \min_{\bm{\beta} \in \mathbb{R}^{dm}} \left [ \frac{1}{2n} \sum_{i=1}^{n} (y_{i} - \tilde{\bm{x}}_{i}'\bm{\beta})^{2} + 
 \frac{\sqrt{m}\lambda_{1}}{n} \sum_{j=1}^{d} \| \hat{\bm{\Sigma}}_{j}^{1/2} \bm{\beta}_{G_{j}} \|_{E} \right ], \label{glasso}
\end{equation}
where recall that $\| \cdot \|_{E}$ denotes the Euclidean norm. 
Therefore, the estimator $\hat{g}$ is computed by solving the group Lasso problem (\ref{glasso}), and
we call $\hat{g}$ the group Lasso estimator. The group Lasso estimator is known to be groupwise sparse. In the present context, this means that some of additive components are estimated as zero functions.

Some comments are in order. 

\begin{remark}
\label{rem1}
In the group Lasso, there is no need to use common basis functions for all $j$; i.e., we may use different basis functions for different $j$. 
To make the notation simpler (e.g. to avoid the extra index $j$ to $\psi_{1},\dots,\psi_{m}, m$ and $\mathcal{G}_{m}$ etc.), we write the group Lasso procedure as it is. 
\end{remark}

\begin{remark}[Computation] 
Because the proposed method is a combination of two commonly used methods, it  can be implemented by using standard statistical software packages. 
In this sense, implementation of the proposed method is simple. 
\end{remark}

\begin{remark}[Other options for the first step procedure]
Although we primarily focus on the group Lasso for the first step variable selection, it is possible to use other variable selection methods available in the 
literature. For instance, the nonparametric independence screening (NIS) method proposed in \cite{FFS11} is known to be a computationally effective way of screening variables. 
However, in view of  (\ref{ibound}),
to obtain a performance bound on the second step estimator, a suitable bound on the magnitude of missed components $\| \sum_{j \in T^{*} \backslash \hat{T}} g_{j}^{*} \|_{2}^{2}$ is required, and at the moment it is not known whether NIS ensures a reasonable bound on it. 
A preferable feature of the group Lasso is that under certain regularity conditions it gives reasonable bounds on both $|\hat{T}^{0} \backslash T^{*} |$ and $\| \sum_{j \in T^{*} \backslash \hat{T}^{0} }g_{j}^{*} \|_{2}^{2}$ (see Section 4). 
\end{remark}

\begin{remark}[Other options for the second step procedure]
There is an alternative second step estimator of $g^{*}$, namely, the sieve least squares estimator applied to the selected additive components: 
\begin{equation}
(\check{c},\check{g}_{j}, j \in \hat{T}) := \arg \min_{c \in \mathbb{R}, g_{j} \in \mathcal{G}_{m}, j \in \hat{T}} \left [ \frac{1}{2n} \sum_{i=1}^{n} \{ y_{i} - c-  {\textstyle \sum}_{j \in \hat{T}} g_{j}(z_{ij}) \}^{2} \right ],
\label{sl}
\end{equation}
subject to the restrictions $\sum_{i=1}^{n} g_{j}(z_{ij}) = 0, \ \forall j \in \hat{T}$ (note: Remark \ref{rem1} applies to this case). Let $\check{g} := \sum_{j \in \hat{T}} \check{g}_{j}$. It is expected that a similar conclusion (to $\tilde{g}$) holds for this estimator.
In terms of estimation accuracy, it is difficult to judge which is theoretically better. 
To make the paper focused, we restrict our attention to $\tilde{g}$ and not make a formal study of $\check{g}$, but compare their finite sample performance by simulations. 
In our limited simulation experiments, $\tilde{g}$ outperforms $\check{g}$ (see Table \ref{table1} ahead), which is a (partial) motivation of studying $\tilde{g}$. 
\end{remark}

\section{Simulation experiments}

This section reports simulation experiments that evaluate the finite sample performance of the estimators.
The estimators under consideration are the group Lasso (GL) estimator defined by (\ref{glasso2}), 
the sieve least square estimator applied to the variables selected by the group Lasso (called GL-SL estimator) defined by (\ref{sl}) with $\hat{T} = \hat{T}^{0}$, 
the penalized least squares estimator applied to the variables selected by the group Lasso (called GL-PL estimator) defined by (\ref{pl}) with $\hat{T}=\hat{T}^{0}$, the penalized least squares estimator with known true support (called ORACLE estimator) defined by (\ref{pl}) with $\hat{T}=T^{*}$,
the \cite{MVB09} estimator (called MGB estimator).  The MGB estimator is defined by a minimizer to the least square criterion function subject to the penalty (\ref{mvbpenalty}) with $\tilde{\lambda}_{3}=0$. The choice $\tilde{\lambda}_{3}=0$ is not theoretically optimal, but what \cite{MVB09} actually proposed in practice is this estimator, 
so in these experiments we take $\tilde{\lambda}_{3}=0$.

To implement the group Lasso, we have to determine basis functions and the penalty level $\lambda_{1}$. 
We use cubic B-splines with four evenly distributed internal knots (so $m=7$). To choose the penalty level, we use an AIC type criterion.
Let $\hat{g}_{\lambda_{1}}$ denote the GL estimator of $g^{*}$ with penalty level $\lambda_{1}$. We choose the optimal penalty level that minimizes  the criterion
\begin{equation*}
\text{AIC}(\lambda_{1}) = n \log ( {\textstyle \sum}_{i=1}^{n} (y_{i} - \bar{y} - \hat{g}_{\lambda_{1}})^{2}/n) + 2 m | \hat{T}_{\lambda}^{0} |,
\end{equation*}
where $\hat{T}_{\lambda_{1}}^{0} := \{ j \in \{ 1, \dots, d \} : \| \hat{g}_{j,\lambda_{1}} \|_{2,n} \neq 0 \}$. Certainly there are other options to choose the penalty level $\lambda_{1}$, such as the cross validation. Here we use the AIC because it is intuitive and straightforward to implement. 
To compute GL estimates, we use the package {\tt grplasso} in {\tt R}. To compute GL-PL estimates and ORACLE estimates, we use the package {\tt mgcv} in {\tt R} in which the smoothing parameters are automatically optimized according to GCV (by default). See \cite{W06}. Comparison with the MGB estimator is not a standard task since its performance depends on the multiple tuning parameters. To guarantee a fair comparison, according to a preliminary simulation work, 
 we prepared  a set of candidates values for $(\tilde{\lambda}_{1},\tilde{\lambda}_{2})$ and evaluated the performance of the MGB estimator for each $(\tilde{\lambda}_{1},\tilde{\lambda}_{2})$. The set of candidate values is given by 
 \begin{multline*}
 \{ (\tilde{\lambda}_{1},\tilde{\lambda}_{2}) : \tilde{\lambda}_{1} = \check {\lambda}_{1} \times \lambda_{\max} /n, \\
\check{\lambda}_{1} \in \{ 0.12,0.08,0.04,0.02 \}, \tilde{\lambda}_{2} \in \{ 0.05,0.02,0.01,0.005 \} \},
 \end{multline*}
 where $\lambda_{\max}$ is computed by the {\tt lambdamax} option in the {\tt grplasso} package when the minimization problem is transformed to the group Lasso problem.

Each estimator is evaluated by the empirical mean square error (EMSE). Let $\mu_{i} := c^{*} + g^{*}(\bm{z}_{i})$ and for a generic estimator $(\hat{c},\hat{g})$ of $(c^{*},g^{*})$, let $\hat{\mu}_{i} := \hat{c} + \hat{g}(\bm{z}_{i})$. Then, the EMSE is defined as 
\begin{equation*}
\text{EMSE} := \mathrm{E} [n^{-1} \sum_{i=1}^{n} (\hat{\mu}_{i} - \mu_{i})^{2}].
\end{equation*}
For GL and MGB estimators, we compute the average numbers of  numbers of variables selected (NV), false positives (FP) and false negatives (FN). 

The number of Monte Carlo repetitions is $500$. 
We consider the case where $n = 400$ and $d=1,000$.
The explanatory variables $\bm{z}_{i} = (z_{i1},\dots,z_{id})'$ are generated as: $z_{ij} = (w_{ij} + tu_{i})/(1+t)$ for $j=1,\dots,d$, where $u_{i},w_{i1},\dots,w_{id}$ are i.i.d. uniform random variables on $[0,1]$. The parameter $t$ controls correlation between variables, i.e., a larger $t$ implies a larger correlation. 
Three cases $t = 0,0.5$ or $1$ are considered. 
In what follows, let $g_{1}(z) = z, g_{2}(z) = (2z-1)^{2}, g_{3}(z) = \sin (2\pi z)/(2-\sin (2\pi z))$ and $g_{4}(z) = 0.1 \sin (2 \pi z)+0.2 \cos (2\pi z) + 0.3 \sin^{2}(2 \pi z) + 0.4 \cos^{3} (2\pi z) + 0.5 \sin^{4} (2\pi z)$. 
We consider two models.
\begin{description}
\item[Model 1] $y_{i} = 5 g_{1}(z_{i1}) + 3g_{2}(z_{i2}) + 4 g_{3}(z_{i3}) + 6 g_{4}(z_{i4}) + \sqrt{1.74} \epsilon_{i}, \ \epsilon_{i} \sim N(0,1)$. 
\item[Model 2] $y_{i} = 3.5 g_{1}(z_{i1}) + 2.1 g_{2}(z_{i2}) + 2.8 g_{3}(z_{i3}) + 4.2 g_{4}(z_{i4}) + 3.5 g_{1}(z_{i5}) + 2.1 g_{2}(z_{i6}) + 2.8 g_{3}(z_{i7}) + 4.2 g_{4}(z_{i8})+ \sqrt{1.74} \epsilon_{i}, \ \epsilon_{i} \sim N(0,1)$.
\end{description}
The coefficients in model 2 are adjusted in such a way that the variance of  the conditional mean of $y_{i}$ given $\bm{z}_{i}$ is roughly the same as in model 1. These designs are essentially adapted from \cite{MVB09}. 

The simulation results are given in Tables \ref{table1} and \ref{table2}. 
Table \ref{table2} shows the performance of the MGB estimator with the tuning parameters chosen in such a way that the EMSE is minimized, and hence the EMSEs in Table \ref{table2} should be understood as the ideal EMSEs of the MGB estimator. Overall the MGB estimator, with the tuning parameters chosen in such a way that the EMSE is minimized, includes too many redundant variables. This feature is consistent with the simulation study in \cite{FFS11}. 

In model 1, in which the number of nonzero additive components is small ($s^{*} = 4$) and each nonzero additive component has a relatively large signal, the variable selection by the group Lasso works well, and hence the GL-PL estimator performs strictly better than the ideal MGB estimator in the EMSE in all cases.

In model 2, in which  the number of nonzero additive components is large ($s^{*} = 8$) and each nonzero additive component has a relatively small signal (compared with model 1), the performance of the GL-PL deteriorates, especially when $t=1$. When $t=1$, that is, the correlation among $\bm{z}_{i}$ is high, it is difficult to detect the nonzero additive components correctly, and the group Lasso on average does not work very well and the EMSE of the GL-PL estimator is worse than the MGB estimator. However this better performance of the MGB estimator is at the cost of selecting many redundant additive components: on average it includes $75$ redundant additive components. 
It turns out that the performance of the MGB estimator is sensitive to the value of $\tilde{\lambda}_{1}$ and not to $\tilde{\lambda}_{2}$. 
Table \ref{table3} shows the performance of the MGB estimator in model 2 with $t=1$ and $\tilde{\lambda}_{2} = 0.05$, and with different values of $\tilde{\lambda}_{1}$ (the best EMSE among all candidate $(\tilde{\lambda}_{1},\tilde{\lambda}_{2})$ in model 2 with $t=1$  is achieved at $\tilde{\lambda}_{1}=0.02 \times \lambda_{\max}/n, \tilde{\lambda}_{2} = 0.05$, which is the reason why we focus on the $\tilde{\lambda}_{2}=0.05$ case). Increasing $\check{\lambda}_{1} = 0.02$ to $\check{\lambda}_{1} = 0.04$ makes the number of false positives small, on average from $75$ to $9$, but makes the EMSE worse, from $0.528$ to $0.921$. 
Taking this into account, we may see that the GL-PL works reasonable well. 

\begin{table}
\begin{threeparttable}
\caption{Simulation results\tnote{} \label{table1}}
\begin{tabular}{c|cccc|c|c|c}
\hline
 & \multicolumn{4}{c|}{GL} & GL-SL & GL-PL & ORACLE  \\
Case & NV & FP & FN & EMSE & EMSE & EMSE & EMSE  \\
\hline
Model 1 ($t=0$) & 5.07 & 1.07 &	0.00 & 1.333 & 0.398 & 0.160 & 0.110  \\
 & (1.16) & (1.16) & (0.00) & (0.201) & (0.113) & (0.067) & (0.031)   \\
\hline
Model 1 ($t=0.5$) & 5.01 & 1.02 & 0.01 & 0.874 & 0.236 & 0.167 & 0.115  \\
 & (1.28) & (1.28) & (0.09) & (0.161) & (0.100) & (0.073) & (0.029)  \\ 
\hline 
Model 1 ($t=1$) & 5.25 & 1.50 & 0.25 & 1.083 & 0.375 & 0.305 & 0.120  \\
 & (1.96) & (1.72) & (0.52) & (0.234) & (0.239) & (0.247) & (0.031)  \\
\hline
Model 2 ($t=0$) & 10.37 & 2.52 & 0.15 & 2.113 & 0.605 & 0.332 & 0.200  \\
 & (2.20) &	(2.06) & (0.38) & (0.336) &	(0.145) & (0.135) & (0.045)  \\ 
\hline
Model 2 ($t=0.5$) & 8.95 & 1.93 & 0.97 & 1.562 & 0.531 & 0.425 & 0.201  \\
 & (2.31) & (1.83) & (0.87) & (0.30) & (0.180)& (0.187) & (0.043)  \\ 
\hline
Model 2 ($t=1$) & 6.11 & 1.24 & 3.13 & 1.796 & 0.996 & 0.945 & 0.203  \\
& (1.86) & (1.42) & (0.80) & (0.24) & (0.144) & (0.154) & (0.043)  \\ 
\hline 
\end{tabular}
\begin{tablenotes}
\item ``GL''  refers to  the group Lasso estimator, ``GL-SL''  to the group Lasso + sieve least squares estimator, ``GL-PL''  to the group Lasso + penalized least squares estimator, ``ORACLE'' to the penalized least squares estimator with known true support, ``NV''  to the number of selected variables, ``FP'' to the false positive, ``FN'' to the false negative, and ``EMSE'' refers to the empirical mean square error. Standard deviations are given in parentheses.
\end{tablenotes}
\end{threeparttable}
\end{table}

\begin{table}
\begin{threeparttable}
\caption{Simulation results (continued)\tnote{} \label{table2}}
\begin{tabular}{c|cccc}
\hline
 & \multicolumn{4}{|c}{MGB} \\
Case &  NV & FP & FN & EMSE \\
\hline
Model 1 ($t=0$)  & 17.86 & 13.86 & 0.00 & 0.357  \\
 &  (6.08) & (6.08) & (0.00)  & (0.068)  \\
\hline
Model 1 ($t=0.5$) & 13.01 & 9.01 & 0.00 & 0.361 \\
 &  (4.85) & (4.85)  & (0.00) & (0.069) \\ 
\hline 
Model 1 ($t=1$) & 70.19 & 66.19 & 0.00 & 0.349 \\
 & (9.90) & (9.90)  & (0.00) & (0.059) \\
\hline
Model 2 ($t=0$) & 62.53 & 54.53 & 0.00 & 0.553  \\
 & (11.77) & (11.77) & (0.00) & (0.080) \\ 
\hline
Model 2 ($t=0.5$)& 44.33 & 36.33 & 0.00 & 0.532 \\
 & (11.11) & (11.11) & (0.06) & (0.081) \\ 
\hline
Model 2 ($t=1$) & 83.00 & 75.01 & 0.01 & 0.528 \\
& (10.14) & (10.14) & (0.08) & (0.070) \\ 
\hline 
\end{tabular}
\begin{tablenotes}
\item ``MGB'' refers  to the (ideal) \cite{MVB09} estimator, ``NV''  to the number of selected variables, ``FP'' to the false positive, ``FN'' to the false negative, and ``EMSE'' refers to the empirical mean square error. Standard deviations are given in parentheses.
\end{tablenotes}
\end{threeparttable}
\end{table}

\begin{table}
\begin{threeparttable}
\caption{Simulation results for the MGB estimator in model 2 with $t=1$ and $\tilde{\lambda}_{2} = 0.05$, and different values of $\tilde{\lambda}_{1}$ \tnote{} \label{table3}}
\begin{tabular}{c|cccc}
\hline
 & $\check{\lambda}_{1} = 0.12$ &$\check{\lambda}_{1} = 0.08$  & $\check{\lambda}_{1} = 0.04$   & $\check{\lambda}_{1} = 0.02$  \\
 \hline 
 NV & 5.14 (0.99)  & 7.82 (1.94) & 16.77 (4.51) & 83.00 (10.14) \\
 FP & 0.27 (0.55) & 1.41 (1.59)  & 8.89 (4.48) & 75.01 (10.14) \\
 FN & 3.14 (0.78) & 1.60 (1.00) & 0.12 (0.35) & 0.01 (0.08) \\
 EMSE & 2.167 (0.146) & 1.624 (0.145) & 0.921 (0.125) & 0.528 (0.070) \\
 \hline 
\end{tabular}
\begin{tablenotes}
\item  ``NV''  refers to the number of selected variables, ``FP'' to the false positive, ``FN'' to the false negative, and ``EMSE'' refers to the empirical mean square error. Standard deviations are given in parentheses.
\end{tablenotes}
\end{threeparttable}
\end{table}

\section{Theoretical study} 

\subsection{Basic conditions}

In this section, we introduce basic conditions commonly used in the analysis of the first and second step estimators. 
\begin{description}
\item[(C1)] (Restriction on the data generating process) $\{ (y_{i},\bm{z}_{i}')' : i=1,2,\dots \}$ are i.i.d. where the pair $(y_{1},\bm{z}_{1}')'$ satisfies the model (\ref{model}).
\item[(C2)] (Restriction on the (conditional) distribution of $u_{1}$) The distribution of $u_{1}$ is such that either:
\begin{enumerate}
\item[(a)] the support of $u_{1}$ is bounded, or;
\item[(b)] $u_{1} | \bm{z}_{1} \sim N(0,\sigma_{u}(\bm{z}_{1})^{2})$ and $\sigma_{u}(\bm{z}_{1}) \leq \sigma_{u}$ almost surely for some constant $\sigma_{u}$ independent of $n$. 
\end{enumerate} 
\item[(C3)] (Restrictions on the distribution of $\bm{z}_{1}$)
\begin{enumerate}
\item[(i)] The support of $\bm{z}_{1}$ is $[0,1]^{d}$.
\item[(ii)] Let $q_{j}$ denote the density of $z_{1j}$ for each $1 \leq j \leq d$. Then, $q_{j}$ is bounded away from zero on $[0,1]$ uniformly over $1 \leq j \leq d$, i.e., there exists a positive constant $c_{q}$ such that $c_{q} \leq q_{j}$ on $[0,1]$ for all $1 \leq j \leq d$. 
\end{enumerate}
\item[(C4)] (Restriction on smoothness of the additive components) $g_{j}^{*} \in \mathcal{G}$ for all $j \in T^{*}$, where $\mathcal{G} = W_{2}^{\nu}([0,1])$ for some positive integer $\nu$.
\item[(C5)] (Preliminary restrictions on $d$ and $s^{*}$) $d \geq n$, $\log d/n^{2\nu/(2\nu+1)} \to 0$ and $1 \leq s^{*} \leq n$. 
\end{description}

Condition (C1) is a standard assumption. Condition (C2) needs an explanation. 
It turns out that the key property to our rate results in Theorems \ref{thm1} and \ref{thm2} (and indeed to those in \cite{KY10}, \cite{RWY10} and \cite{STS11} as well)
is the normal concentration property (around its mean and given $\bm{z}_{1},\dots,\bm{z}_{n}$) of a random variable of the form $\sup_{t \in \mathcal{T}}  \sum_{i=1}^{n} u_{i} t_{i} $ where $\mathcal{T}$ is a bounded and countable subset of $\mathbb{R}^{n}$ ($\mathcal{T}$ typically depends on $\bm{z}_{1},\dots,\bm{z}_{n}$). In fact, condition (C2) is a primitive sufficient condition that ensures this normal concentration property. 
See Appendix C for more discussion on this condition.
Condition (C3) is standard in the series estimation literature \citep[see e.g.][]{N97}. Condition (C4) restricts the smoothness property of each additive component $g_{j}^{*}$. We exclude the case that $\nu$ is fractional. Condition (C5) is a preliminary restriction on the growth rate of $d$. To make the technical argument simpler, we here assume that $d \geq n$. Because our primal concern is on the ``$d \gg  n$'' case, this restriction does not bind.
The second part of condition (C5) is to restrict $d$ not to grow too fast. The last part of condition (C5) is a natural restriction on $s^{*}$.

\subsection{A generic bound on the second step estimator}

In this section, we present a generic bound on the second step estimator. Although we primarily focus on to use the group Lasso as a first step procedure, 
the result of this section holds for any variable selection method satisfying the high level condition stated below. 

We first prepare some notation.
Let $\mathcal{G}_{j} := \{ g_{j} \in \mathcal{G} : \mathrm{E}[ g_{j}(z_{1j}) ] = 0 \}$. 
For a subset  $T \subset \{ 1,\dots, d \}$, define
\begin{equation*}
\alpha (T) := \inf \left \{ \alpha > 0 : \sum_{j \in T} \| g_{j} \|_{2}^{2} \leq \alpha  \| \sum_{j \in T} g_{j}   \|_{2}^{2}, \ \forall g_{j} \in \mathcal{G}_{j} \ (j \in T) \right \}.
\end{equation*}
The quantity $\alpha (T)^{-1}$ is an analogue of sparse minimum eigenvalues to the infinite dictionary case. It is clear that when $z_{1j}, j \in T$ are independent, $\alpha (T) = 1$, so $\alpha (T)$ measures the dependence among variables $z_{1j}, j \in T$ (recall that each function in $\mathcal{G}_{j}$ is centered such that $\mathrm{E}[g_{j}(z_{1j})] = 0$).  Such a quantity appears in other papers on estimation of high dimensional additive models \citep{KY10,STS11}. Observe that $\alpha (T) \geq 1$ for any non-empty $T \subset \{1,\dots,d \}$.

We introduce a high level condition on $\hat{T}$. Put 
\begin{equation*}
\delta := \delta_{n} := \max \left \{ n^{-\nu/(2\nu+1)}, \sqrt{\frac{\log d}{n}} \right \}.
\end{equation*} 

\begin{description}
\item[(C6)] (Restriction on the set $\hat{T}$) $n^{1/2(2\nu+1)} \delta \alpha (T^{*} \cup \hat{T}) | T^{*} \cup \hat{T} |   = o_{p}(1)$.
\end{description}

Note that under condition (C5), $$n^{1/2(2\nu+1)} \delta = \max \{ n^{-(2\nu-1)/2(2\nu+1)}, \sqrt{\log d/n^{2\nu/(2\nu+1)}} \} \to 0.$$ Condition (C6) requires that $\alpha (T^{*} \cup \hat{T})$ and $| T^{*} \cup \hat{T} |$ are not too large.  
In the canonical case in which $\alpha (T^{*} \cup \hat{T}) \lesssim_{p} 1$ and $| \hat{T} | \lesssim_{p} s^{*}$, condition (C6) is satisfied if $s^{*} = o [ \min \{ n^{(2\nu-1)/ 2(2\nu+1)}, \sqrt{n^{2\nu/(2\nu+1)}/\log d} \}]$. We shall comment that, even when $T^{*}$ were known, a condition analogous to (C6) is needed to obtain 
a reasonable bound on the estimator, so we believe that, as long as $| \hat{T} |$ is stochastically not overly large compared with $s^{*}$, condition (C6) is a reasonable restriction. It will be shown that, when the group Lasso is used as a first step procedure, $| \hat{T}^{0} | \lesssim_{p} s^{*}$.

We are now in position to state the main theorem of this paper. 

\begin{theorem}
\label{thm1}
Assume conditions (C1)-(C6). Take $\lambda_{2}$ such that $\lambda_{2} \geq A_{2,u,\nu} n^{-\nu/(2\nu+1)}$,
where $A_{2,u,\nu}$ is some positive constant depending only on the distribution of $u_{1}$ and the smoothness index $\nu$. Then, we have 
\begin{align*}
&\| \tilde{g} - g^{*} \|^{2}_{2} + \lambda_{2}^{2} \sum_{j \in \hat{T}} I(\tilde{g}_{j})^{2} \\
&\lesssim_{p} \max \Big \{   \alpha (T^{*} \cup \hat{T})| T^{*} \cap \hat{T} | n^{-2\nu/(2\nu+1)}, \alpha (T^{*} \cup \hat{T}) | \hat{T} \backslash T^{*} | \delta^{2}, n^{-2\nu/(2\nu+1)} \| g^{*} \|_{2}^{2} \\
&\qquad \qquad \lambda_{2}^{2} \sum_{j \in T^{*} \cap \hat{T}} I(g_{j}^{*})^{2}, \| {\textstyle \sum}_{j \in T^{*} \backslash \hat{T}} g_{j}^{*} \|_{2}^{2}, n^{-2\nu/(2\nu+1)} \sum_{j \in T^{*} \backslash \hat{T}} I(g_{j}^{*})^{2} \Big \}.
\end{align*}
In particular, in the canonical case in which (i) $\alpha (T^{*} \cup \hat{T}) \lesssim_{p} 1$; (ii) $\| g^{*} \|_{2}^{2} \lesssim s^{*}$; (iii) $\sum_{j \in T^{*}} I(g_{j}^{*})^{2} \lesssim s^{*}$, for $\lambda_{2} \geq A_{2,u,\nu} n^{-\nu/(2\nu+1)}$, 
we have 
\begin{equation}
\| \tilde{g} - g^{*} \|^{2}_{2} + \lambda_{2}^{2} \sum_{j \in \hat{T}} I(\tilde{g}_{j})^{2} \lesssim_{p} \max \left \{ s^{*} \lambda_{2}^{2}, | \hat{T} \backslash T^{*} | \delta^{2}, \| {\textstyle \sum}_{j \in T^{*} \backslash \hat{T}} g_{j}
^{*} \|_{2}^{2} \right \}. \label{ebound}
\end{equation} 
\end{theorem}

\begin{remark}
In principle,  it is possible to state the theorem in a non-asymptotic manner; however, to make the exposition clear, we state the theorem as it is. 
\end{remark}

Interestingly, $\lambda_{2}$ can be taken independent of $d$ despite the random fluctuation of $\hat{T}$. This is in contrast to the fact that, e.g. in \cite{KY10,RWY10}, penalty levels (on smoothness) should scale as $\log d$ as $d \to \infty$. 

This theorem characterizes the effect of the first step variable selection in an explicit manner: in (\ref{ebound}), (i) the first term $s^{*} \lambda_{2}^{2}$ reflects the oracle rate; 
(ii) the second term $| \hat{T} \backslash T^{*} | \delta^{2}$ reflects the cost of selecting redundant components; (iii) the third term $ \| {\textstyle \sum}_{j \in T^{*} \backslash \hat{T}} g_{j}
^{*} \|_{2}^{2}$ reflects the magnitude of missed components. We will investigate the behaviors of these terms when the group Lasso is used as a first step procedure.

\subsection{Properties of the group Lasso}

In this section, we collect the statistical properties (namely the convergence rate and the model selection property) of the group Lasso estimator $\hat{g}$ used as a first step estimator. 
Although such properties have been well studied in the literature especially for the parametric regression case \citep{NR08,B08,RLLW09,HZ10,WH10,HHW10,LPTV10,NRWY10}, 
we could not find results that we exactly need in the very present setting, in particular an explicit scaling condition on the triple $(d,s^{*},m)$ that guarantees the statistical properties. For the sake of completeness, we state here these properties. Their proofs are found in Appendix. 

We begin with introducing restrictions on basis functions. 
\begin{description}
\item[(C7)] (Restrictions on basis functions used in the first step estimation)
\begin{enumerate}
\item[(a)] $\sup_{z \in [0,1]} \| (\psi_{1}(z),\dots,\psi_{m}(z))' \|_{E} = O(m^{1/2})$.
\item[(b)] $\mathrm{E}[\tilde{\bm{x}}_{1G_{j}} \tilde{\bm{x}}_{1G_{j}}'] = \bm{I}_{m}$ for all $1 \leq j \leq d$.
\item[(c)] $\inf_{g \in \mathcal{G}_{m}^{T^{*}}} \| g^{*} - g \|_{2}^{2} \lesssim s^{*} m^{-2\nu}$, where $\mathcal{G}_{m}^{T^{*}} := \{ g: \mathcal{Z} \to \mathbb{R} : g(\bm{z}) = \sum_{j \in T^{*}} g_{j}(z_{j}) \ (\bm{z} = (z_{1},\dots,z_{d})'), \ g_{j} \in \mathcal{G}_{m} \ (j \in T^{*}) \}$.
\end{enumerate}
\end{description}

We refer to \cite{N97} for some basic materials on series estimation. Condition (C7)-(a) is satisfied for splines and Fourier bases. 
Condition (C7)-(b) is a normalization condition, and does not lose any generality as long as we are concerned with the analysis of the statistical properties of the group Lasso estimator. Condition (C7)-(c) corresponds to condition (C4) and is thought to be a reasonable restriction. Consider, for instance, $\psi_{1},\dots,\psi_{m}$ are spline functions of degree $(\nu+1)$ on $[0,1]$ with equidistant knots. 
By Corollary 6.26 of \cite{S07}, there exists a 
$g^{m} = \sum_{j \in T^{*}} g_{j}^{m} \in \mathcal{G}_{m}^{T^{*}}$ such that 
$\sum_{j \in T^{*}} \|  g^{*}_{j} - g^{m}_{j} \|_{2}^{2} \lesssim m^{-2\nu} \sum_{j \in T^{*}} I(g^{*}_{j})^{2}$. Because $\mathrm{E}[g^{*}_{j}(z_{1j})] = 0$, $g^{m}_{j}$ may be taken such that $\mathrm{E}[ g^{m}_{j}(z_{1j}) ] = 0$. Therefore, letting for a subset $T \subset \{ 1,\dots, d \}$, 
\begin{equation*}
\beta (T) := \inf \left \{ \beta > 0 : \| \sum_{j \in T} g_{j} \|_{2}^{2} \leq \beta \sum_{j \in T} \| g_{j} \|_{2}^{2}, \ \forall g_{j} \in \mathcal{G}_{j} \ (j \in T) \right \},
\end{equation*}
if $\sum_{j \in T^{*}} I(g_{j})^{2} \lesssim s^{*}$ and $\beta (T^{*}) \lesssim 1$, then $\| g^{*} - g^{m} \|_{2}^{2} \leq \beta(T^{*}) \sum_{j \in T^{*}} \|  g^{*}_{j} - g^{m}_{j} \|_{2}^{2} \lesssim \beta(T) m^{-2\nu} \sum_{j \in T^{*}} I(g^{*}_{j})^{2} \lesssim s^{*} m^{-2\nu}$. 
The restriction that $\sum_{j \in T^{*}} I(g^{*}_{j})^{2} \lesssim s^{*}$ is reasonable. Trivial examples in which $\beta (T^{*}) \lesssim 1$ are the case that $s^{*} \lesssim 1$ or the case that $z_{1j}, j \in T^{*}$ are independent. Conditions similar to $\beta (T^{*}) \lesssim 1$ appear in other papers such as \cite{KY10}.

We now start to investigate the statistical properties of the group Lasso estimator.
To this end, we prepare some notation. 
Define the event 
\begin{equation*}
\Omega_{0} := \{ \| \hat{\bm{\Sigma}}_{j}^{1/2} - \bm{I}_{m} \| \leq 0.5, \ 1 \leq \forall j \leq d \}.
\end{equation*}
We will later give a sufficient condition under which $\mathrm{P}(\Omega_{0}) \to 0$, which means that, with probability approaching one, all $\hat{\bm{\Sigma}}_{j}$ are ``well behaved'' in the sense that they are not too much deviated from their population values. 

Define the set 
\begin{equation*}
\mathbb{C} := \{ \bm{\alpha} \in \mathbb{R}^{dm} : \sum_{j \in (T^{*})^{c}} \| \bm{\alpha}_{G_{j}} \|_{E} \leq 21 \sum_{j \in T^{*}} \| \bm{\alpha}_{G_{j}} \|_{E} \}.
\end{equation*}
The set $\mathbb{C}$ is a cone, i.e., for any $\bm{\alpha} \in \mathbb{C}$ and $c > 0$, $c \bm{\alpha} \in \mathbb{C}$. It consists of vectors $\bm{\alpha} \in \mathbb{R}^{dm}$ such that the coordinates of $\bm{\alpha}$ in the set $T^{*}$ are dominant. Such cones of dominant coordinates play an important role in the analysis of penalization methods for high dimensional statistical models. 
Define the {\em $\mathbb{C}$-restricted eigenvalue} of $\hat{\bm{\Sigma}}^{1/2}$ by 
\begin{equation*}
\hat{\kappa} := \min_{\bm{\alpha} \in \mathbb{S}^{dm-1} \cap \mathbb{C}}  \| \hat{\bm{\Sigma}}^{1/2} \bm{\alpha} \|_{E}.
\end{equation*}
Restricted eigenvalues are originally introduced by \cite{BRT09} for the Lasso formulation. While the minimum eigenvalue of $\hat{\bm{\Sigma}}$ is always zero when $dm \geq n$, $\hat{\kappa}$ can be positive with a high probability as long as the corresponding restricted eigenvalue of the population matrix $\bm{\Sigma}$ is bounded away from zero (see Lemma B.4 in Appendix B). 

Put $\check{\bm{x}}_{iG_{j}} := \hat{\bm{\Sigma}}_{j}^{-1/2} \tilde{\bm{x}}_{iG_{j}}$, where $\hat{\bm{\Sigma}}_{j}^{-1/2}$ is interpreted as the generalized inverse of $\hat{\bm{\Sigma}}_{j}^{1/2}$ if it is singular. If $\hat{\bm{\Sigma}}_{j} = \bm{U} \bm{D} \bm{U}'$ denotes the spectral decomposition of 
$\hat{\bm{\Sigma}}_{j}$ where $\bm{U}$ is an $m \times m$ orthogonal matrix and $\bm{D}$ is an $m \times m$ diagonal matrix with diagonal entries $d_{1} \geq \cdots \geq d_{l} > 0 = d_{l+1}= \cdots = d_{m}$, then $\hat{\bm{\Sigma}}_{j}^{-1/2}$ is given by $\hat{\bm{\Sigma}}_{j}^{-1/2} = \bm{U} \diag \{ d_{1}^{-1/2},\dots,d_{l}^{-1/2},0,\dots,0\} \bm{U}'$. Invoke that on the event $\Omega_{0}$, all $\hat{\bm{\Sigma}}_{j}$ are nonsingular. Define the random variable 
\begin{equation*}
\Lambda := \max_{1 \leq j \leq d} \left \| \sum_{i=1}^{n} u_{i} \check{\bm{x}}_{iG_{j}} / \sqrt{m} \right \|_{E}.
\end{equation*}
This random variable plays a ``threshold''value for $\lambda_{1}$.

We state a preliminary bound on $\hat{g}$ in terms of $\| \cdot \|_{2,n}$. 
\begin{proposition}
\label{prop1}
On the event $\{ \lambda_{1} \geq 2 \Lambda \} \cap \{ \hat{\kappa} > 0 \} \cap \Omega_{0}$, we have
\begin{equation*}
\| g^{*} - \hat{g} \|^{2}_{2,n}  \leq 2 \inf_{g \in \tilde{\mathcal{G}}_{m}^{T^{*}}}  \| g^{*} - g \|^{2}_{2,n} + C_{2} \frac{s^{*}m \lambda_{1}^{2}}{\hat{\kappa}^{2} n^{2}},
\end{equation*}
where $C_{2}$ is a universal constant and $\tilde{\mathcal{G}}_{m}^{T^{*}} := \{ g : \mathcal{Z} \to \mathbb{R} : g (\bm{z}) = \sum_{j \in T^{*}} \sum_{k=1}^{m} \beta_{jk} (\psi_{k}(z_{j}) - \bar{\psi}_{jk}) \ (\bm{z} = (z_{1},\dots,z_{d})'), \ \beta_{jk} \in \mathbb{R} \ (j \in T^{*}; 1 \leq k \leq m) \}$. 
\end{proposition}

To state the model selection property of the group Lasso estimator, we need another concept, namely, {\em group sparse eigenvalues}.  
For any subset $T \subset \{ 1,\dots,d \}$, let $\mathbb{S}^{dm-1}_{T} := \{ \bm{\alpha} \in \mathbb{R}^{dm} :  \bm{\alpha}_{G_{T^{c}}} = \bm{0} \} \cap \mathbb{S}^{dm-1}$. Define the {\em $s$-th group sparse maximum eigenvalue} of $\hat{\bm{\Sigma}}^{1/2}$ by
\begin{equation*}
\hat{\phi}_{\max}(s) := \max_{| T | \leq s, \bm{\alpha} \in \mathbb{S}_{T}^{dm-1}} \| \hat{\bm{\Sigma}}^{1/2} \bm{\alpha} \|_{E}.
\end{equation*}
The next proposition gives a preliminary bound on $\hat{s}$, the number of  components selected by the group Lasso estimator $\hat{g}$: 
$\hat{s} := | \hat{T}^{0} | = | \{ j \in \{ 1,\dots, d \} : \| \hat{g}_{j} \|_{2,n} \neq 0 \} |$. 

\begin{proposition}
\label{prop2}
Let $\hat{C} :=  3 n\| g^{*} - \hat{g} \|_{2,n}/(\sqrt{s^{*}m}\lambda_{1})$ and $\mathcal{S} := \{ s \in \{ 1,\dots, d \}: s > 2\hat{C}^{2} \hat{\phi}_{\max}(s)^{2}s^{*} \}$. On the event $\{ \lambda_{1} \geq 2  \Lambda \vee 0 \} \cap \Omega_{0}$, 
we have
\begin{equation*}
\hat{s} \leq \hat{C}^{2}  [ \min_{s \in \mathcal{S}} \hat{\phi}_{\max}(s)^{2}] s^{*}.
\end{equation*}
\end{proposition}

Propositions \ref{prop1} and \ref{prop2} are deterministic statements, and they do not use any stochastic argument.
In order to bound stochastic orders of $\| g^{*} - \hat{g} \|_{2,n}$ and $\hat{s}$, we have to determine: (i) conditions that ensure $\mathrm{P}(\Omega_{0}) \to 1$; (ii)
a value of $\lambda_{1}$ such that $\lambda_{1} \geq 2  \Lambda$ with probability approaching one; (iii) conditions that ensure desired asymptotic behaviors of $\hat{\kappa}$ and $\hat{\phi}_{\max}(s)$; (iv) an stochastic order of the approximation error $\inf_{g \in \tilde{\mathcal{G}}_{m}^{T^{*}}} \| g^{*} - g \|^{2}_{2,n}$. Lemmas B.1-B.5 in Appendix B are concerned with these issues. We shall comment that while the proofs of Propositions \ref{prop1} and \ref{prop2} are a direct adaptation of the corresponding proofs in the Lasso case, the proofs of Lemmas B.1-B.4 are not the case because the fact that the size ($m$) of each group goes to infinity brings a subtle technical issue. Given Propositions \ref{prop1} and \ref{prop2}, and Lemmas B.1-B.5 in Appendix B, we obtain the following theorem. 

\begin{theorem}
\label{thm2}
Assume conditions (C1)-(C5) and (C7). 
Assume further that $s^{*},m,d$ and $n$ obey the growth condition $(s^{*})^{2}m\log (d \vee n)/n \to 0$;  $\phi_{\max} (s) := \max_{| T | \leq s, \bm{\alpha} \in \mathbb{S}_{T}^{dm-1}} \| \bm{\Sigma}^{1/2} \bm{\alpha} \|_{E} \lesssim 1$ for some sequence $s=s_{n}$ such that $s/s^{*} \to \infty$; $\kappa := \min_{\bm{\alpha} \in \mathbb{S}^{dm-1} \cap \mathbb{C}} \| \bm{\Sigma}^{1/2} \bm{\alpha} \|_{E} \gtrsim 1$; and $\| g^{*} \|_{2}^{2} \lesssim s^{*}$. Take $\lambda_{1} \geq A_{1,u} \sqrt{n}(1+\sqrt{\log d/m})$ with constant $A_{1,u}$ given in Lemma B.2 in Appendix B and $m \gtrsim n^{1/(2\nu+1)}$. Then: 
\begin{equation*}
\| g^{*} - \hat{g} \|^{2}_{2,n} \lesssim_{p} \frac{s^{*} m \lambda_{1}^{2}}{n^{2}}, \ \hat{s} \lesssim_{p} s^{*}.
\end{equation*}
In particular, if $m \asymp n^{1/(2\nu+1)}$ and 
\begin{equation}
\lambda_{1} \asymp \max \left \{ \sqrt{n} , \sqrt{ \frac{n\log d}{m} } \right \}, \label{lambda1}
\end{equation}
then we have $\| g^{*} - \hat{g} \|^{2}_{2,n} \lesssim_{p} s^{*} \delta^{2}$.
\end{theorem}

\begin{proof}
See Appendix B. 
\end{proof}


\begin{remark}
When $m \asymp n^{\nu/(2\nu+1)}$, the order of $d$ allowed is $\log d=o\{ n^{2\nu/(2\nu+1)}/(s^{*})^{2} \}$. 
If $\log d \asymp n^{a}$ and $s^{*} \asymp n^{b}$ for some $a,b \geq 0$, the region that $(a,b)$ is allowed is $\{ (a,b) : a, b \geq 0, a+2b < 2\nu/(2\nu+1) \}$. It is interesting to note that this region is large when $\nu$ is large, i.e., the additive components are more smooth. This indicates that the more smooth the additive components are, the larger $d$ and $s^{*}$ can be. 
\end{remark}

We consider the magnitude of missed components $\| \sum_{j \in T^{*} \backslash \hat{T}^{0}} g_{j}^{*} \|_{2}^{2}$. To this end, 
for a subset $T \subset \{ 1,\dots, d \}$, define the $T$-sparse minimal eigenvalue $\hat{\phi}_{\min}(T)$ of $\hat{\bm{\Sigma}}$ by 
\begin{equation*}
\hat{\phi}_{\min}(T) := \min_{\bm{\alpha} \in \mathbb{S}_{T}^{dm-1}} \| \hat{\bm{\Sigma}}^{1/2} \bm{\alpha} \|_{E}.
\end{equation*}
We also need a slightly stronger approximation property than condition (C7)-(c). 
\begin{description}
\item[(C7)] (c)' There exists a $g^{m} = \sum_{j \in T^{*}} g_{j}^{m} \in \mathcal{G}_{m}^{T^{*}}$ such that $\max_{T \subset T^{*}} \| \sum_{j \in T} (g_{j} - g_{j}^{m}) \|_{2}^{2} \lesssim s^{*} m^{-2\nu}$.
\end{description}

\begin{corollary}[Magnitude of missed components]
\label{cor1}
Assume the same conditions as in Theorem \ref{thm2} with condition (C7)-(c) replaced by (C7)-(c)'.
Assume further that $\hat{\phi}_{\min} (T^{*} \cup \hat{T}) \gtrsim_{p} 1$. Then, we have $\| \sum_{j \in T^{*} \backslash \hat{T}^{0}} g_{j}^{*} \|_{2}^{2} \lesssim_{p} s^{*} m\lambda_{1}^{2}/n^{2}$.
\end{corollary}

This corollary clarifies sufficient conditions under which the magnitude of missed components is not larger than the bound on $\| \hat{g} - g^{*} \|_{2,n}^{2}$. 
When $m \asymp n^{1/(2\nu +1)}$ and $\lambda_{1}$ is (\ref{lambda1}), then, under the conditions of Corollary \ref{cor1},
$| \hat{T}^{0} \backslash T^{*} | \lesssim_{p} s^{*}$ and $\| \sum_{j \in T^{*} \backslash \hat{T}^{0}} g_{j}^{*} \|_{2}^{2} \lesssim_{p} s^{*} \delta^{2}$. In that case,  the second step estimator $\tilde{g}$ with $\hat{T} = \hat{T}^{0}$ and $A_{2,u,\nu} n^{-\nu/(2\nu+1)} \leq \lambda_{2} \lesssim \delta$ satisfies that $\| \tilde{g} - g^{*} \|_{2}^{2} + \lambda_{2}^{2} \sum_{j \in \hat{T}} I(\tilde{g}_{j}) \lesssim_{p} s^{*} \delta^{2}$. This bound holds in general cases in which $\hat{T}^{0}$ may fail to recover $T^{*}$. 
If it happens that $\hat{T}^{0} = T^{*}$ with probability approaching one, the estimator $\tilde{g}$ 
(with $\lambda_{2} \asymp n^{-\nu/(2\nu+1)}$) enjoys the exact oracle rate $s^{*} n^{-2\nu/(2\nu+1)}$. As long as taking $\lambda_{2} \asymp n^{-2\nu/(2\nu+1)}$, the estimator $\tilde{g}$ adapts to both situations. 

Sufficient conditions for the perfect model selection are found in, e.g., Theorem 2 of \cite{RLLW09}. Unfortunately, their condition (39) does not cover our 
choice of the penalty level $\lambda_{1}$. Note that the correspondence between their notation (left) and our notation (right) is: $p=d, d_{n} = m$ and $\lambda_{n} = \sqrt{m}\lambda_{1}/n$. However, a careful inspection of their proof shows that their condition (39) can be replaced by a weaker condition that 
there exists some  constant $C > 0$ such that
\begin{equation}
\frac{\lambda_{n}^{2} n}{d_{n} \vee \log p} > C \ (\text{in their notation}), \ \text{or} \ \frac{m\lambda_{1}^{2}}{n(m \vee \log d)} > C \ (\text{in our notation}), \label{pmodel}
\end{equation}
which covers our choice of the penalty level $\lambda_{1}$. To see this, observe that their condition (39) is used only to ensure (85) in their appendix, which can be replaced by (in their notation) $\mathrm{P}(\max_{j \in S^{c}} \| \hat{g}_{j} - \mu_{j} \| > \delta/2) \to 0$, or equivalently $\mathrm{P}(\max_{j \in S^{c}}  \| Z_{j} \| \geq \lambda_{n} \delta/2) \to 0$. By using first the union bound and then Theorem 7.1 of \cite{L01} (the Gaussian concentration inequality) similarly to the proof of our Lemma B.2 in Appendix B, it is shown that condition (\ref{pmodel}) is sufficient for that $\mathrm{P}(\max_{j \in S^{c}}  \| Z_{j} \| \geq \lambda_{n} \delta/2) \to 0$.

\subsection{Comparison with other work} 

In this section, we briefly state connections and differences of the proposed method from some existing estimation methods for high dimensional additive models.
It must be said that the literature on high dimensional additive models is now growing; so it is beyond the scope of this paper to review all the existing methods in details.

In \cite{MVB09},  the penalized least squares estimator defined by the solution to the following minimization problem is proposed:{\footnotesize
\begin{equation*}
\min_{c \in \mathbb{R},  g_{j} \in \mathcal{G}, 1 \leq j \leq d} \left [ \frac{1}{2n}\sum_{i=1}^{n} (y_{i} - c -   \sum_{j=1}^{d} g_{j}(z_{ij}))^{2} + 
\sum_{j=1}^{d}\left \{ \tilde{\lambda}_{1} \sqrt{\| g_{j} \|_{2,n}^{2} + \tilde{\lambda}_{2} I(g_{j})^{2}} + \tilde{\lambda}_{3} I(g_{j})^{2} \right \} \right],
\end{equation*}
 }
where the term $\| g_{j} \|_{2,n}$ controls sparsity while the term $I(g_{j})$ controls smoothness of $g_{j}$.  
\cite{KY10} and \cite{RWY10} considered a doubly penalized estimation method similar to \cite{MVB09} but in a (more general) reproducing kernel Hilbert space (RKHS) formulation.  
\cite{STS11} further analyzed the \cite{MVB09} method and established a faster convergence rate than \cite{MVB09} did in a more general setting. 
The method proposed in this paper is thought to be a method that splits such a ``double penalization'' into two steps, and 
 intends to remove a shrinkage bias caused by simultaneously penalizing sparsity and smoothness.

\cite{HHW10} proposed a two-step estimation method different from ours. 
Their proposal is to construct consistent estimators of the additive components at the first step, and then to use these consistent estimators to 
apply the adaptive group Lasso, which is a modification of the adaptive Lasso \citep{Z06} to the group Lasso case.
In particular, they proposed to use the group Lasso for the first step estimation. 
To be precise, under the notation of Section 2.2, let $\hat{\bm{\beta}}^{0}$ denote the solution to the group Lasso problem (\ref{glasso}) with $\hat{\bm{\Sigma}}_{j}^{1/2}$ replaced by $\bm{I}_{m}$, and use this group Lasso estimator to construct the weights: $w_{j} := 1/\| \hat{\bm{\beta}}_{G_{j}}^{0} \|_{E}$ (we agree that $1/0 = \infty$).  The adaptive group Lasso estimator is then defined by $\hat{g}^{A}(\bm{z}) = \sum_{j=1}^{d} \hat{g}^{A}_{j}(z_{j}), \ \hat{g}^{A}_{j}(z_{j}) = \sum_{k=1}^{m}\hat{\beta}^{A}_{jk}(\psi_{k}(z_{j}) - \bar{\psi}_{jk})$, where 
\begin{equation*}
\hat{\bm{\beta}}^{A} := \arg \min_{\bm{\beta} \in \mathbb{R}^{dm}} \left [\frac{1}{2n} \sum_{i=1}^{n} (y_{i}-\tilde{\bm{x}}_{i}'\bm{\beta})^{2} + 
\lambda_{A} \sum_{j=1}^{d} w_{j} \| \bm{\beta}_{G_{j}} \|_{E} \right].
\end{equation*}
The adaptive group Lasso can be seen as a post model selection estimator. In fact, since $w_{j} = \infty$ when $\| \hat{\bm{\beta}}_{G_{j}} \|_{E} = 0$, the adaptive group Lasso problem reduces to
\begin{equation*}
\min_{\bm{\beta}_{G_{j}}, j \in \check{T}} \left [\frac{1}{2n} \sum_{i=1}^{n} (y_{i}-\sum_{j \in \check{T}} \tilde{\bm{x}}_{iG_{j}}'\bm{\beta}_{G_{j}})^{2} + \lambda_{A} \sum_{j \in \check{T}} w_{j} \| \bm{\beta}_{G_{j}} \|_{E} \right],
\end{equation*}
where $\check{T} := \{ j \in \{ 1,\dots, d \} : \| \hat{\bm{\beta}}^{0}_{G_{j}} \|_{E} \neq 0 \}$. 
Therefore, their estimation method is similar to ours in some respect. 
Besides the similarity, however, there is a notable difference between two methods. \cite{HHW10} intend to select {\em correctly} the nonzero additive components with probability approaching one by using the group Lasso penalty at both the first and second steps, while our method intends to ensure sparsity and smoothness of the estimator. A point to be noticed is that the analysis of \cite{HHW10} substantially depends on the assumption that the non-zero additive components are well separated from zero in the $L_{2}$-sense, which is, as argued in Introduction, significantly restrictive from a theoretical point of view, and it is this assumption why the adaptive group Lasso estimator can achieve the exact oracle rate in their analysis. Therefore, from a strict theoretical sense, their theoretical result is not directly comparable to ours (and \cite{MVB09}). 

\section{Conclusion}

In this paper we have investigated the two-step estimation of high dimensional additive models. Especially, we have derived a generic performance bound on the second step estimator, and studied the overall performance when the group Lasso is used as a first step variable selection. 
Diving the overall estimation procedure into two steps enables us to help shrinkage bias caused by the double penalization strategy, and we believe that 
the theoretical and numerical properties explored in this paper are useful suggestions to practical applications.


\section{Proof of Theorem \ref{thm1}}

The proof of Theorem \ref{thm1} uses the next technical lemma. Its proof is based on a use of empirical process techniques. Define $\mu_{g_{j}} := \mathrm{E}[ g_{j} (z_{1j}) ]$. 

\begin{lemma}
\label{lem1}
Assume conditions (C1)-(C5).
Then, there exist a positive constant $C_{u,\nu}$ depending only on the distribution of $u_{1}$ and the smoothness index $\nu$, and a positive constant $c_{q,\nu}$ depending only on $c_{q}$ (given in condition (C3)) and $\nu$  such that the following holds: for any sequence of nonempty subsets $T = T_{n} \subset \{ 1,\dots, d \}$ and for any sequence of constants $\epsilon = \epsilon_{n} \to 0$ such that $\epsilon \geq C_{u,\nu} n^{-\nu/(2\nu+1)}$, 
we have, with probability approaching one:  
{\footnotesize
\begin{align*}
&(i) \ \left | \frac{1}{n} \sum_{i=1}^{n} u_{i} g_{j}(z_{ij}) \right | \leq \max \{ \epsilon, C_{1} \sqrt{\log (s \vee n)/n} \} \sqrt{\| g_{j} \|_{2}^{2} + \epsilon^{2} I(g_{j})^{2}}, \ \forall g_{j} \in \mathcal{G}, \ \forall j \in T; \\
&(ii) \  \left | \frac{1}{n} \sum_{i=1}^{n} \{ g_{j}(z_{ij}) - \mu_{g_{j}} \} \right | \leq \max \{ \epsilon, C_{1} \sqrt{\log (s \vee n)/n} \} \sqrt{\| g_{j} \|_{2}^{2} + \epsilon^{2} I(g_{j})^{2}}, \ \forall g_{j} \in \mathcal{G}, \ \forall j \in T; \\
&(iii) \
| \| g \|_{2,n}^{2} - \| g \|_{2}^{2} | \leq  c_{q,\nu} \epsilon^{-1/(2\nu)} \max \{ \epsilon, \delta \} \left [ \sum_{j=1}^{d} \sqrt{ \| g_{j} \|_{2}^{2} + \epsilon^{2} I(g_{j})^{2}} \right ]^{2}, \ \forall g = \sum_{j=1}^{d} g_{j}, \ g_{j} \in \mathcal{G},
\end{align*}
}
where $s := s_{n} := | T |$ and $C_{1} > 0$ is a universal constant.
\end{lemma}

\begin{proof}[Proof of Lemma \ref{lem1}]
See Section Appendix A. 
\end{proof}

\begin{proof}[Proof of Theorem \ref{thm1}]
We first point out that because of the restriction $\sum_{i=1}^{n} g_{j}(z_{ij}) = 0$, by a standard argument, 
we may assume that $\hat{c} = c^{*} = \mathrm{E}[ y_{1} ] = 0$ for the analysis of $\tilde{g}$. 

Let $C_{u,\nu},c_{q,\nu}$ and $C_{1}$ be the constants given in Lemma \ref{lem1}. Take $\epsilon = \epsilon_{n} = C_{u,\nu} n^{-\nu/(2\nu+1)}$ and $\lambda_{2} \geq \sqrt{2} \epsilon$. 
Define the events 
\begin{align*}
&\Omega_{1} := \text{event (i) of Lemma \ref{lem1} with $T=T^{*}$}, \\
&\Omega_{2} := \text{event (i) of Lemma \ref{lem1} with $T=(T^{*})^{c}$}, \\
&\Omega_{3} := \text{event (ii) of Lemma \ref{lem1} with $T=T^{*}$}, \\
&\Omega_{4} := \text{event (ii) of Lemma \ref{lem1} with $T=(T^{*})^{c}$}, \\
&\Omega_{5} := \text{event (iii) of Lemma \ref{lem1}}.
\end{align*}
In what follows, we go through the proof on the events $\cap_{k=1}^{5} \Omega_{k}$. Note that the probability of this event goes to one. 
Because $s^{*} = | T^{*} | \leq n$, we may assume that $C_{1}\sqrt{\log (s^{*} \vee  n)/n} = C_{1} \sqrt{\log n/n} \leq \epsilon$ in events (i) and (ii) of Lemma \ref{lem1} with $T=T^{*}$.  
Let $\varrho := \varrho_{n} := \max \{ \epsilon, C_{1}\sqrt{\log d/n} \}$. Invoke that $\varrho \asymp \delta$. We may assume that $\varrho \leq 1$.  
For $j \notin \hat{T}$, we agree that $\tilde{g}_{j} \equiv 0$. 

Because of the optimality of $\tilde{g}$, 
\begin{equation*}
\frac{1}{2n}\sum_{i=1}^{n} (y_{i} - \sum_{j \in \hat{T}} \tilde{g}_{j}(z_{ij}))^{2} + \lambda_{2}^{2} \sum_{j \in \hat{T}} I(\tilde{g}_{j})^{2} 
\leq \frac{1}{2n}\sum_{i=1}^{n} (y_{i} - \sum_{j \in \hat{T}} g_{j}^{*}(z_{ij}))^{2} + \lambda_{2}^{2} \sum_{j \in \hat{T}} I(g^{*}_{j})^{2}.
\end{equation*}
Then, 
using the relation 
\begin{equation*}
(y_{i} - g(\bm{z}_{i}))^{2} 
=u_{i}^{2} + 2 u_{i}(g^{*}(\bm{z}_{i}) - g(\bm{z}_{i})) + (g^{*}(\bm{z}_{i}) - g(\bm{z}_{i}))^{2},
\end{equation*}
we have (note that $\tilde{g}_{j} \equiv 0$ for $j \notin \hat{T}$)
\begin{align*}
&\frac{1}{2} \| g^{*} - \tilde{g} \|_{2,n}^{2} + \lambda_{2}^{2} \sum_{j \in \hat{T}} I(\tilde{g}_{j})^{2} \\
&\leq   \sum_{j \in \hat{T}} \left [ \frac{1}{n} \sum_{i=1}^{n} u_{i}  \{ \tilde{g}_{j}(z_{ij}) - g_{j}^{*}(z_{ij}) \} \right ] + \lambda_{2}^{2} \sum_{j \in \hat{T}} I(g^{*}_{j})^{2} + \frac{1}{2} \| {\textstyle \sum}_{j \in T^{*} \backslash \hat{T}} g^{*}_{j} \|_{2,n}^{2}.
\end{align*}
Using the facts that $\sqrt{a+b} \leq \sqrt{a} + \sqrt{b}, ab \leq 0.5(a^{2} + b^{2})$ and $(a+b)^{2} \leq 2(a^{2}+b^{2})$,
{\footnotesize
\begin{align*}
&\sum_{j \in T^{*} \cap \hat{T}} \left [ \frac{1}{n} \sum_{i=1}^{n} u_{i}  \{ \tilde{g}_{j}(z_{ij}) - g_{j}^{*}(z_{ij}) \} \right ] \\
&\leq \epsilon \sum_{j \in T^{*} \cap \hat{T}} \sqrt{\| \tilde{g}_{j} - g_{j}^{*} \|_{2}^{2} + \epsilon^{2} I(\tilde{g}_{j} - g_{j}^{*})^{2}} \quad (\because \Omega_{1}) \\
&\leq \epsilon \sum_{j \in T^{*} \cap \hat{T}} \| \tilde{g}_{j} - g_{j}^{*} \|_{2} + \epsilon^{2} \sum_{j \in T^{*} \cap \hat{T}} I(\tilde{g}_{j} - g_{j}^{*}) \\
&\leq \epsilon \sqrt{ | T^{*} \cap \hat{T} |\sum_{j \in T^{*} \cap \hat{T}} \| \tilde{g}_{j} - g_{j}^{*} \|^{2}_{2} } + 0.5 \epsilon^{2} | T^{*} \cap \hat{T} | + 0.5 \epsilon^{2}\sum_{T^{*} \cap \hat{T}} I(\tilde{g}_{j} - g_{j}^{*})^{2}  \\
&\leq \epsilon \sqrt{ | T^{*} \cap \hat{T} |\sum_{j \in T^{*} \cap \hat{T}} \| \tilde{g}_{j} - g_{j}^{*} \|^{2}_{2} } +  0.5 \epsilon^{2} | T^{*} \cap \hat{T} | +  \epsilon^{2}\sum_{j \in T^{*} \cap \hat{T}} I(\tilde{g}_{j})^{2} +  \epsilon^{2}\sum_{j \in T^{*} \cap \hat{T}}I(g_{j}^{*})^{2}.
\end{align*}
}
For any fixed $b > 0$, 
\begin{align*}
\epsilon \sqrt{ | T^{*} \cap \hat{T} |\sum_{j \in T^{*} \cap \hat{T}} \| \tilde{g}_{j} - g_{j}^{*} \|^{2}_{2} } &= \sqrt{ 2b \epsilon^{2}| T^{*} \cap \hat{T} | \times \frac{1}{2b}\sum_{j \in T^{*} \cap \hat{T}} \| \tilde{g}_{j} - g_{j}^{*} \|^{2}_{2} } \\
&\leq b\epsilon^{2}| T^{*} \cap \hat{T} | + \frac{1}{4b} \sum_{j \in T^{*} \cap \hat{T}} \| \tilde{g}_{j} - g_{j}^{*} \|^{2}_{2}.
\end{align*}
Similarly, we have
{\footnotesize
\begin{align*}
&\sum_{j \in \hat{T} \backslash T^{*}} \left [ \frac{1}{n} \sum_{i=1}^{n} u_{i}  \{ \tilde{g}_{j}(z_{ij}) - g_{j}^{*}(z_{ij}) \} \right ] \\
&=\sum_{j \in \hat{T} \backslash T^{*}} \left [ \frac{1}{n} \sum_{i=1}^{n} u_{i}   \tilde{g}_{j}(z_{ij}) \right ] \\
&\leq \varrho \sum_{j \in \hat{T} \backslash T^{*}} \sqrt{\| \tilde{g}_{j}  \|_{2}^{2} + \epsilon^{2} I(\tilde{g}_{j})^{2}} \quad (\because \Omega_{2}) \\
&\leq \varrho \sum_{j \in \hat{T} \backslash T^{*}} \| \tilde{g}_{j}  \|_{2} + \varrho \epsilon \sum_{j \in \hat{T} \backslash T^{*}} I(\tilde{g}_{j}) \\
&\leq \varrho \sqrt{ | \hat{T} \backslash T^{*} |\sum_{j \in \hat{T} \backslash T^{*}} \| \tilde{g}_{j} \|^{2}_{2} } + 0.5 \varrho^{2} | \hat{T} \backslash T^{*} | + 0.5 \epsilon^{2}\sum_{j \in \hat{T} \backslash T^{*}} I(\tilde{g}_{j} )^{2}  \\
&\leq \frac{1}{4b} \sum_{j \in  \hat{T}\backslash T^{*}} \| \tilde{g}_{j} \|^{2}_{2} +  (b + 0.5) \varrho^{2} | \hat{T} \backslash T^{*} | + 0.5  \epsilon^{2}\sum_{j \in \hat{T} \backslash T^{*}} I(\tilde{g}_{j})^{2}.
\end{align*}
}
Thus, we have 
\begin{multline}
\frac{1}{2} \| g^{*} - \tilde{g} \|_{2,n}^{2} + (\lambda_{2}^{2} - \epsilon^{2}) \sum_{j \in \hat{T}} I(\tilde{g}_{j})^{2} \\
\leq \frac{1}{4b} \sum_{j \in  \hat{T}} \| \tilde{g}_{j} - g^{*}_{j} \|^{2}_{2}+ (b + 0.5) (\epsilon^{2} | T^{*} \cap \hat{T} |
+ \varrho^{2} | \hat{T} \backslash T^{*} | )  \\
 + (\lambda_{2}^{2} + \epsilon^{2})\sum_{j \in T^{*} \cap \hat{T}}I(g_{j}^{*})^{2}   + \frac{1}{2} \| {\textstyle \sum}_{j \in T^{*} \backslash \hat{T}} g^{*}_{j} \|_{2,n}^{2}. \label{ineq1}
\end{multline}
 
Recall the definition of $\alpha (T)$. 
Invoke now that 
\begin{align*}
\sum_{j \in  \hat{T}} \| \tilde{g}_{j} - g_{j}^{*} \|^{2}_{2} &\leq \sum_{j \in  T^{*} \cup \hat{T}} \| \tilde{g}_{j} - g_{j}^{*} \|^{2}_{2} \\
&\leq 2\sum_{j \in T^{*} \cup \hat{T}} \| (\tilde{g}_{j} - \mu_{\tilde{g}_{j}})- g_{j}^{*} \|^{2}_{2} + 2\sum_{j \in \hat{T}} \mu_{\tilde{g}_{j}}^{2} \quad (\mu_{g_{j}} := \mathrm{E}[g_{j}(z_{1j})]) \\
&\leq 2 \alpha (T^{*} \cup \hat{T}) \| (\tilde{g} - \mu_{\tilde{g}})- g^{*} \|^{2}_{2} +  2\sum_{j \in \hat{T}} \mu_{\tilde{g}_{j}}^{2} \\
&\leq 2 \alpha (T^{*} \cup \hat{T}) \| \tilde{g} - g^{*} \|^{2}_{2} + 2\sum_{j \in \hat{T}} \mu_{\tilde{g}_{j}}^{2}. 
\end{align*}
Because of the restriction $\sum_{i=1}^{n} \tilde{g}(z_{ij}) = 0$, $\mu_{\tilde{g}_{j}} = - n^{-1}\sum_{i=1}^{n} (\tilde{g}_{j}(z_{ij}) - \mu_{\tilde{g}_{j}})$, so that, because of the event $\Omega_{3}$, for all $j \in T^{*} \cap \hat{T}$, 
\begin{align*}
\mu^{2}_{\tilde{g}_{j}} 
&\leq \epsilon^{2} \| \tilde{g}_{j} \|_{2}^{2} + \epsilon^{4} I(\tilde{g}_{j})^{2} \\
&\leq 2\epsilon^{2} \| \tilde{g}_{j} - g^{*}_{j} \|_{2}^{2}
+ 2\epsilon^{2} \| g_{j}^{*} \|_{2}^{2} + \epsilon^{4} I(\tilde{g}_{j})^{2},
\end{align*}
while because of the event $\Omega_{4}$, for all $j \in \hat{T} \backslash T^{*}$, $\mu^{2}_{\tilde{g}_{j}} 
\leq \varrho^{2} \| \tilde{g}_{j} \|_{2}^{2} + \varrho^{2} \epsilon^{2} I(\tilde{g}_{j})^{2}$.
Thus, noting that $\varrho \geq \epsilon$, 
\begin{multline*}
(1-\max \{ 4 \epsilon^{2}, 2 \varrho^{2} \} ) \sum_{j \in \hat{T}} \| \tilde{g}_{j} - g_{j}^{*} \|^{2}_{2} \\
\leq 2 \alpha (T^{*} \cup \hat{T}) \| \tilde{g} - g^{*} \|^{2}_{2}+4\epsilon^{2} \sum_{j \in T^{*} \cap \hat{T}} \| g_{j}^{*} \|_{2}^{2} + 2\varrho^{2} \epsilon^{2} \sum_{j \in \hat{T}} I(\tilde{g}_{j})^{2},
\end{multline*}
so that for $n$ large enough (such that $ \max \{ 4 \epsilon^{2}, 2 \varrho^{2} \} \leq  0.5$), 
\begin{equation}
\sum_{j \in \hat{T}} \| \tilde{g}_{j} - g_{j}^{*} \|^{2}_{2} \leq 4 \alpha (T^{*} \cup \hat{T}) \| \tilde{g} - g^{*} \|^{2}_{2}+8\epsilon^{2} \sum_{j \in T^{*} \cap \hat{T}} \| g_{j}^{*} \|_{2}^{2} +  4\varrho^{2} \epsilon^{2} \sum_{j \in \hat{T}} I(\tilde{g}_{j})^{2}. \label{ineq2}
\end{equation}
Substituting (\ref{ineq2}) into (\ref{ineq1}), we obtain 
\begin{align}
&\frac{1}{2} \| g^{*} - \tilde{g} \|_{2,n}^{2} + \left ( \lambda_{2}^{2} - \epsilon^{2} - \frac{\varrho^{2}\epsilon^{2}}{b} \right ) \sum_{j \in \hat{T}} I(\tilde{g}_{j})^{2} \notag \\
&\leq \frac{\alpha (T^{*} \cup \hat{T})}{b} \| \tilde{g} - g^{*} \|^{2}_{2}+ (b + 0.5) (\epsilon^{2} | T^{*} \cap \hat{T} | + \varrho^{2} | \hat{T} \backslash T^{*} | )  \notag \\
&\quad +  \frac{2\epsilon^{2}}{b}  \sum_{j \in T^{*} \cap \hat{T}} \| g_{j}^{*} \|_{2}^{2} + (\lambda_{2}^{2} + \epsilon^{2})\sum_{j \in T^{*} \cap \hat{T}}I(g_{j}^{*})^{2} +  \frac{1}{2} \| {\textstyle \sum}_{j \in T^{*} \backslash \hat{T}} g^{*}_{j} \|_{2,n}^{2}. \label{ineq3}
\end{align}

We next consider a lower bound on $\| g^{*} - \tilde{g} \|_{2,n}^{2}$ Observe that 
{\footnotesize
\begin{align*}
&\| g^{*} - \tilde{g} \|_{2,n}^{2} \\
&\geq \| g^{*} - \tilde{g} \|^{2}_{2} - c_{q,\nu} \epsilon^{-1/(2\nu)} \max \{ \epsilon, \delta \} \left [ \sum_{j \in T^{*} \cup \hat{T}} \sqrt{\| g_{j}^{*}-\tilde{g}_{j} \|^{2}_{2} +  \epsilon^{2} I(g_{j}^{*}-\tilde{g}_{j})^{2}} \right]^{2} \quad (\because \Omega_{5})\\
&\geq \| g^{*} - \tilde{g} \|^{2}_{2} - c_{q,\nu}\epsilon^{-1/(2\nu)} \max \{ \epsilon, \delta \} | T^{*} \cup \hat{T} |  \sum_{j \in T^{*} \cup \hat{T}} \{ \| g_{j}^{*}-\tilde{g}_{j} \|^{2}_{2} + \epsilon^{2} I(g_{j}^{*}-\tilde{g}_{j})^{2} \} \\
&\geq \| g^{*}  - \tilde{g} \|^{2}_{2} - c_{q,\nu} \epsilon^{-1/(2\nu)} \max \{ \epsilon, \delta \} | T^{*} \cup \hat{T} | \left \{  \alpha (T^{*} \cup \hat{T})  \| g^{*}-\tilde{g} \|^{2}_{2} + \epsilon^{2} \sum_{j \in T^{*} \cup \hat{T}} I(g_{j}^{*}-\tilde{g}_{j})^{2} \right \} \\
&\geq (1-\hat{c}_{1}) \| g^{*} - \tilde{g} \|^{2}_{2} -2 \hat{c}_{2}  \epsilon^{2} \sum_{j \in T^{*}} I(g_{j}^{*})^{2} -2 \hat{c}_{2} \epsilon^{2} \sum_{j \in \hat{T}} I(\tilde{g}_{j})^{2},
\end{align*}
}
where $\hat{c}_{1} :=c_{q,\nu} \epsilon^{-1/(2\nu)} \max \{ \epsilon, \delta \} | T^{*} \cup \hat{T} | \alpha (T^{*} \cup \hat{T})$ and $\hat{c}_{2} := c_{q,\nu}\epsilon^{-1/(2\nu)} \max \{ \epsilon, \delta \} | T^{*} \cup \hat{T}|$. 
Substituting this inequality to (\ref{ineq3}), we have
\begin{align*}
&\left ( \frac{1}{2}-\frac{\hat{c}_{1}}{2} - \frac{\alpha (T^{*} \cup \hat{T})}{b} \right ) \| g^{*} - \tilde{g} \|^{2}_{2} + \left ( \lambda_{2}^{2} - \epsilon^{2} - \frac{\varrho^{2} \epsilon^{2}}{b} - \hat{c}_{2}\epsilon^{2} \right ) \sum_{j \in \hat{T}} I(\tilde{g}_{j})^{2} \\
&\leq (b + 0.5) (\epsilon^{2} | T^{*} \cap \hat{T} | + \varrho^{2} | \hat{T} \backslash T^{*} | )+  \frac{2\epsilon^{2}}{b}  \sum_{j \in T^{*} \cap \hat{T}} \| g_{j}^{*} \|_{2}^{2}  \\
&\quad + (\lambda_{2}^{2} + \epsilon^{2}+\hat{c}_{2}\epsilon^{2})\sum_{j \in T^{*} \cap \hat{T}}I(g_{j}^{*})^{2}+\frac{1}{2} \| {\textstyle \sum}_{j \in T^{*} \backslash \hat{T}} g^{*}_{j} \|_{2,n}^{2} + \hat{c}_{2}\epsilon^{2} \sum_{j \in T^{*} \backslash \hat{T}}I(g_{j}^{*})^{2}.
\end{align*}

We wish to bound the term $\| \sum_{j \in T^{*} \backslash \hat{T}} g_{j}^{*} \|_{2,n}^{2}$. Observe that 
{\footnotesize
\begin{align*}
&\| {\textstyle \sum}_{j \in T^{*} \backslash \hat{T}} g_{j}^{*} \|_{2,n}^{2} \\
&\leq \|{\textstyle \sum}_{j \in T^{*} \backslash \hat{T}} g_{j}^{*} \|^{2}_{2} + c_{q,\nu} \epsilon^{-1/(2\nu)} \max \{ \epsilon, \delta \} \left [ \sum_{j \in T^{*} \backslash \hat{T}} \sqrt{\| g_{j}^{*} \|^{2}_{2} +  \epsilon^{2} I(g_{j}^{*})^{2}} \right]^{2} \quad (\because \Omega_{5})\\
&\leq \|{\textstyle \sum}_{j \in T^{*} \backslash \hat{T}} g_{j}^{*} \|^{2}_{2} + c_{q,\nu} \epsilon^{-1/(2\nu)} \max \{ \epsilon, \delta \} | T^{*} \backslash \hat{T} |  \sum_{j \in T^{*} \backslash \hat{T}} \{ \| g_{j}^{*} \|^{2}_{2} + \epsilon^{2} I(g_{j}^{*})^{2} \} \\
&\leq \|{\textstyle \sum}_{j \in T^{*} \backslash \hat{T}} g_{j}^{*} \|^{2}_{2} + c_{q,\nu} \epsilon^{-1/(2\nu)} \max \{ \epsilon, \delta \} | T^{*} \backslash \hat{T} | \left \{  \alpha (T^{*} \backslash \hat{T})  \| {\textstyle \sum}_{j \in T^{*} \backslash \hat{T}} g_{j}^{*} \|^{2}_{2} + \epsilon^{2} \sum_{j \in T^{*} \backslash \hat{T}} I(g_{j}^{*})^{2} \right \} \\
&\leq (1+\hat{c}_{3}) \| {\textstyle \sum}_{j \in T^{*} \backslash \hat{T}} g_{j}^{*} \|^{2}_{2} + \hat{c}_{4} \epsilon^{2} \sum_{j \in T^{*}} I(g_{j}^{*})^{2},
\end{align*}
}
where $\hat{c}_{3} := c_{q,\nu} \epsilon^{-1/(2\nu)} \max \{ \epsilon, \delta \} | T^{*} \backslash \hat{T} | \alpha (T^{*} \backslash \hat{T})$ and $\hat{c}_{4} := c_{q,\nu} \epsilon^{-1/(2\nu)} \max \{ \epsilon, \delta \} | T^{*} \backslash \hat{T} |$. Therefore, we have 
{\footnotesize
\begin{align*}
&\left ( \frac{1}{2}-\frac{\hat{c}_{1}}{2} - \frac{\alpha (T^{*} \cup \hat{T})}{b} \right ) \| g^{*} - \tilde{g} \|^{2}_{2} + \left ( \lambda_{2}^{2} - \epsilon^{2} - \frac{\varrho^{2} \epsilon^{2}}{b} - \hat{c}_{2}\epsilon^{2} \right ) \sum_{j \in \hat{T}} I(\tilde{g}_{j})^{2} \\
&\leq (b + 0.5) (\epsilon^{2} | T^{*} \cap \hat{T} | + \varrho^{2} | \hat{T} \backslash T^{*} | )+  \frac{2\epsilon^{2}}{b}  \sum_{j \in T^{*} \cap \hat{T}} \| g_{j}^{*} \|_{2}^{2} + (\lambda_{2}^{2} + \epsilon^{2}+\hat{c}_{2}\epsilon^{2})\sum_{j \in T^{*} \cap \hat{T}}I(g_{j}^{*})^{2} \\
&\quad + \frac{(1+\hat{c}_{3})}{2} \| {\textstyle \sum}_{j \in T^{*} \backslash \hat{T}} g^{*}_{j} \|_{2}^{2} + \left (\hat{c}_{2} + \frac{\hat{c}_{4}}{2} \right ) \epsilon^{2} \sum_{j \in T^{*} \backslash \hat{T}}I(g_{j}^{*})^{2}.
\end{align*}
}

Taking $b=4 \alpha (T^{*} \cup \hat{T}) \geq 4$ and noting that $\varrho \leq 1$, we have 
{\footnotesize
\begin{align*}
&\left ( \frac{1}{4}-\frac{\hat{c}_{1}}{2} \right ) \| g^{*} - \tilde{g} \|^{2}_{2} + \left \{  \lambda_{2}^{2} - \left ( \frac{5}{4} + \hat{c}_{2} \right )\epsilon^{2} \right  \} \sum_{j \in \hat{T}} I(\tilde{g}_{j})^{2} \\
&\leq  \{ 4 \alpha (T^{*} \cup \hat{T}) + 0.5 \} (\epsilon^{2} | T^{*} \cap \hat{T} | + \varrho^{2} | \hat{T} \backslash T^{*} | ) +  \frac{\epsilon^{2}}{2} \| g^{*} \|_{2}^{2} + (\lambda_{2}^{2} + \epsilon^{2}+\hat{c}_{2}\epsilon^{2})\sum_{j \in T^{*} \cap \hat{T} }I(g_{j}^{*})^{2} \\
&\quad + \frac{(1+\hat{c}_{3})}{2} \| {\textstyle \sum}_{j \in T^{*} \backslash \hat{T}} g^{*}_{j} \|_{2}^{2} + \left (\hat{c}_{2} + \frac{\hat{c}_{4}}{2} \right ) \epsilon^{2} \sum_{j \in T^{*} \backslash \hat{T}}I(g_{j}^{*})^{2},
\end{align*}
}
where we have used the inequality 
\begin{equation*}
\frac{2}{b} \sum_{j \in T^{*} \cap \hat{T}} \| g_{j}^{*} \|_{2}^{2} \leq \frac{2}{b} \sum_{j \in T^{*}} \| g_{j}^{*} \|_{2}^{2}
\leq \frac{2\alpha (T^{*})}{b} \| g^{*} \|_{2}^{2} \leq 0.5 \| g^{*} \|_{2}^{2}.
\end{equation*}
Because $\hat{c}_{1} = o_{p}(1), \hat{c}_{2} = o_{p}(1), \hat{c}_{3} = o_{p}(1)$ and $\hat{c}_{4} = o_{p}(1)$ by condition (C6), we have $\hat{c}_{1} \leq 1/4,\hat{c}_{2} \leq 1/2, \hat{c}_{3} \leq 1$ and $\hat{c}_{4} \leq 1$ with probability approaching one. 
Define the event 
\begin{equation*}
\Omega_{6} := \{ \hat{c}_{1} \leq 1/4, \hat{c}_{2} \leq 1/2, \hat{c}_{3} \leq 1, \hat{c}_{4} \leq 1 \}.
\end{equation*}
Recall that $\lambda^{2}_{2} \geq 2 \epsilon^{2}$. Therefore, on the event $\cap_{k=1}^{6} \Omega_{k}$, we have 
\begin{multline*}
\frac{1}{8} \| g^{*} - \tilde{g} \|^{2}_{2} + \frac{\lambda_{2}^{2}}{4}  \sum_{j \in \hat{T}} I(\tilde{g}_{j})^{2} \leq \{ 4 \alpha (T^{*} \cup \hat{T}) + 0.5 \}   (\epsilon^{2} | T^{*} \cap \hat{T} | + \varrho^{2} | \hat{T} \backslash T^{*} | ) \\
+  \frac{\epsilon^{2}}{2}\| g^{*} \|_{2}^{2} + (\lambda_{2}^{2} + 1.5 \epsilon^{2}) \sum_{j \in T^{*} \cap \hat{T}}I(g_{j}^{*})^{2}
+ \| {\textstyle \sum}_{j \in T^{*} \backslash \hat{T}} g_{j}^{*} \|_{2}^{2} +  \epsilon^{2}  \sum_{j \in T^{*} \backslash \hat{T}}I(g_{j}^{*})^{2}.
\end{multline*}
The desired conclusion follows from the fact that $\mathrm{P}( \cap_{k=1}^{6} \Omega_{k} ) \to 1$.
\end{proof}

\section*{Acknowledgments}

The author acknowledges Dr. Lukas Meier for sharing his codes used in  \cite{MVB09} and Dr. Isamu Nagai for helping the numerical experiments. 
Most of the work was done when the author was visiting Department of Economics, MIT. He greatly acknowledges their hospitality. 
This work was supported by the Grant-in-Aid for Young Scientists (B) (22730179) from the JSPS.

\appendix

\section{Proof of Lemma 6.1}

In the proofs below, we agree that $C$ denotes a universal constant, and its value may change from line to line. 
The same rule applies to Appendix B.

\subsection{Preliminary lemmas}

In this section, we collect some preliminary results used in the proof of Lemma 6.1.
We begin with introducing an interpolation inequality by \cite{G67}.

\begin{lemma}[\cite{G67}]
\label{lemA1}
For any $f \in W_{2}^{\nu}([0,1])$ with positive integer $\nu$, 
\begin{equation*}
\| f \|_{\infty} \leq 
\begin{cases}
K \| f \|_{L_{2}(\lambda)}^{(2\nu-1)/(2\nu)} \| f^{(\nu)} \|_{L_{2}(\lambda)}^{1/(2\nu)}, &\ \text{if} \ \| f^{(\nu)} \|_{L_{2}(\lambda)} \neq 0, \\
K \| f \|_{L_{2}(\lambda)}, &\ \text{if} \ \| f^{(\nu)} \|_{L_{2}(\lambda)} = 0,
\end{cases}
\end{equation*}
where $K$ is a constant independent of $f$, and $\| \cdot \|_{L_{2}(\lambda)}$ denotes the $L_{2}$-norm with respect to the Lebesgue measure $\lambda$ on $[0,1]$.
\end{lemma}

We also use the next lemma. For any probability measure $Q$ on $[0,1]$, define the $\nu \times \nu$ matrix 
\begin{equation*}
\bm{\Sigma}_{Q,\nu} := \mathrm{E}_{Z \sim Q} \left [
\begin{pmatrix}
1 \\
Z \\
\vdots \\
Z^{\nu-1}
\end{pmatrix}
 (1,Z,\dots,Z^{\nu-1}) \right ].
\end{equation*}

\begin{lemma}
\label{lemA2}
Let $Q$ be any probability measure on $[0,1]$ such that the matrix $\bm{\Sigma}_{Q,\nu}$ is non-sigular. 
For any $f \in W_{2}^{\nu}([0,1])$ ($\nu$ is a positive integer), there exist functions $f^{[1]}$ and $f^{[2]}$ such that (i) $f=f^{[1]}+f^{[2]}$; 
(ii) $\| f^{[1]} \|_{\infty} \leq \const \times I(f^{[1]})$ (the constant is independent of $f$); (iii) $f^{[2]}$ is a polynomial function on $[0,1]$ of degree $\nu-1$; and (iv) $\int f^{[1]} f^{[2]} dQ =0$. 
\end{lemma}

\begin{proof}[Proof (sketch)]
Lemma \ref{lemA2} is used in \cite{MVB09} but without proof. For the sake of completeness, we provide a sketch of the proof. Take any $f \in W_{2}^{n}([0,1])$. It is standard to see that there exist functions $f^{[1]}$ and $f^{[2]}$ such that $f=f^{[1]} + f^{[2]}$, 
 $\| f^{[1]} \|_{\infty} \leq  I(f^{[1]})$ and $f^{[2]}$ is a polynomial function on $[0,1]$ of degree $\nu-1$ (use Taylor's theorem). 
Let $\tilde{f}^{[1]}$ denote the orthogonal projection (in $L_{2}(Q)$) of $f^{[1]}$ onto the space of all polynomial functions of degree $\nu-1$. Then, using the fact that $\bm{\Sigma}_{Q,\nu}$ is non-singular, by a simple algebra, each coefficient of $z^{k} \ (k=0,\dots,\nu-1)$ in $\tilde{f}^{[1]}$ is bounded by $K \| f^{[1]} \|_{\infty} \leq K I(f^{[1]})$, so that $\| \tilde{f}^{[1]} \|_{\infty} \leq K' I(f^{[1]})$ ($K$ and $K'$ are constants independent of $f$). Replacing 
$f^{[1]}$ by $f^{[1]} - \tilde{f}^{[1]}$ and $f^{[2]}$ by $f^{[2]} + \tilde{f}^{[1]}$, we obtain the desired conclusion.  
\end{proof}

It is standard to see that $\bm{\Sigma}_{Q,\nu}$ is non-singular if the density of $Q$ is bounded away from zero on $[0,1]$. 
The next lemma is due to Corollary 5 of \cite{MVB09}, which is basically deduced from an entropy integral argument and a peeling argument (such techniques are described in Chapter 8 of \cite{vdG00}). Recall that 
$I(f)^{2}:={\displaystyle \int_{0}^{1}} f^{(\nu)}(z)^{2} dz$.

\begin{lemma}[Essentially \cite{MVB09}, Corollary 5]
\label{lemA3}
Let $\xi_{1},\dots,\xi_{n}$ be i.i.d. from a distribution $Q$ on $[0,1]$ such that $\bm{\Sigma}_{Q,\nu}$ is non-singular, and let $\sigma_{1},\dots,\sigma_{n}$ be independent Rademacher random variables independent of $\xi_{1},\dots,\xi_{n}$. Let $\mathrm{E}_{\sigma}[ \cdot ]$ denote the conditional expectation with respect to $\sigma_{1},\dots,\sigma_{n}$ given $\xi_{1},\dots,\xi_{n}$. Then, there exists a positive constant $C_{\nu}$ depending only on $\nu$ such that 
for all $\epsilon \geq C_{\nu} n^{-\nu/(2\nu+1)}$, 
{\footnotesize
\begin{equation*}
\mathrm{E}\left [ \sup_{f \in W_{2}^{\nu}([0,1])} \frac{| n^{-1} \sum_{i=1}^{n} \sigma_{i} f(\xi_{i}) |}{\sqrt{ \| f \|_{2}^{2} + \epsilon^{2} I(f)^{2}}} \right ] \leq C_{\nu} \epsilon, \ 
\mathrm{E}_{\sigma} \left [ \sup_{\substack{f \in W_{2}^{\nu}([0,1]) \\ I(f) \leq 1, \| f \|_{\infty} \leq 1}} \frac{| n^{-1} \sum_{i=1}^{n} \sigma_{i} f(\xi_{i}) |}{\sqrt{ \| f \|_{2,n}^{2} + \epsilon^{2} I(f)^{2}}} \right ] \leq C_{\nu} \epsilon,
\end{equation*}
}
where $\| f \|^{2}_{2,n} := n^{-1} \sum_{i=1}^{n} f(\xi_{i})^{2}$ and $\| f \|_{2}^{2} := \mathrm{E}[ f(\xi_{1})^{2}]$. The conclusion is true when $\sigma_{1},\dots,\sigma_{n}$ are independent standard normal. 
\end{lemma}

\begin{proof}
The first inequality is Corollary 5 of \cite{MVB09}. Note that their $s$, $\alpha$ and $\gamma$ correspond to $s = \nu, \alpha=1-1/(2\nu)$ and $\gamma=2/(2\nu+1)$ in our notation. 
The second inequality can be shown in a similar way. 
\end{proof}

\begin{add}
\label{add1}
It is clear that, under the same conditions of Lemma \ref{lemA2}, for any constant $K > 0$, 
\begin{equation*}
\mathrm{E}_{\sigma} \left [ \sup_{\substack{f \in W_{2}^{\nu}([0,1]) \\ I(f) \leq 1, \| f \|_{\infty} \leq K}} \frac{| n^{-1} \sum_{i=1}^{n} \sigma_{i} f(\xi_{i}) |}{\sqrt{ \| f \|_{2,n}^{2} + \epsilon^{2} I(f)^{2}}} \right ] \leq C_{\nu} \epsilon.
\end{equation*}
\end{add}

The next two lemmas compare the empirical and population $L_{2}$-norms on the class $\mathcal{G}$ uniformly over the distributions of $z_{1j} \ (j \in T)$.

\begin{lemma}
\label{lemA4}
Assume conditions (C1), (C3) and (C4). Let $T$ be any subset of $\{ 1,\dots, d \}$ and $s := | T |$.  Let $C_{\nu}$ be the constant given in Lemma \ref{lemA1}. 
Then, there exists a positive constant $C_{q,\nu}$ depending only on $c_{q}$ (which is given in condition (C3)) and $\nu$ such that, as long as 
\begin{equation*}
C_{\nu} n^{-\nu/(2\nu+1)} \leq \epsilon \leq 1 \ \text{and} \ C_{q,\nu} \epsilon^{-1/(2\nu)} \max \{ \epsilon, \sqrt{\log (s \vee n)/n} \} \leq 0.5,
\end{equation*}
with probability at least $1-(s \vee n)^{-1}$, 
\begin{equation*}
\| g_{j} \|_{2,n}^{2} \leq 1.5\| g_{j} \|_{2}^{2} + 0.5\epsilon^{2} I(g_{j})^{2}, \ \forall g_{j} \in \mathcal{G}, \ \forall j \in T.
\end{equation*}
\end{lemma}

\begin{proof}
In this proof, $C_{q,\nu}$ denotes some positive constant depending only on $c_{q}$ and $\nu$. Its value may change from line to line.
Pick any $j \in T$.
For a constant $\epsilon \in (0,1]$ specified later, define 
\begin{equation*}
Z_{j} := \sup_{g_{j} \in \mathcal{G}} \frac{\| g_{j} \|_{2,n}^{2} - \| g_{j} \|^{2}_{2}}{(\| g_{j} \|_{2}^{2} + \epsilon^{2} I(g_{j})^{2})}.
\end{equation*}
By Lemma \ref{lemA1}, invoke that when $I(g_{j}) \neq 0$, 
\begin{align*}
\| g_{j} \|_{\infty} &\leq C_{q,\nu} \| g_{j} \|_{2}^{(2\nu-1)/(2\nu)} I(g_{j})^{1/(2\nu)} \\
&\leq C_{q,\nu} \| g_{j} \|_{2}^{(2\nu-1)/(2\nu)} (\epsilon I(g_{j}))^{1/(2\nu)} \epsilon^{-1/(2\nu)} \\
&\leq C_{q,\nu} \epsilon^{-1/(2\nu)} \sqrt{ \| g_{j} \|_{2}^{2} + \epsilon^{2} I(g_{j})^{2}},
\end{align*}
and 
\begin{align*}
\mathrm{E}[ g_{j}(z_{1j})^{4}] &\leq \| g_{j} \|_{\infty}^{2} \| g_{j} \|_{2}^{2} \\
&\leq C_{q,\nu} \epsilon^{-1/\nu} (\| g_{j} \|_{2}^{2} + \epsilon^{2} I(g_{j})^{2})^{2}.
\end{align*}
Even if $I(g_{j}) = 0$, these inequalities hold (with a suitable change to the constant $C_{q,\nu}$ if necessary) since $\epsilon \leq 1$.  
Thus, by Massart's (2000) form of Talagrand's (1996)  inequality, for all $t > 0$, with probability at least $1-e^{-t}$, 
\begin{equation*}
Z_{j} \leq 2\mathrm{E}[Z_{j}] + C_{q,\nu} \sqrt{\epsilon^{-1/\nu}t/n} + C_{q,\nu} \epsilon^{-1/\nu} t/n.
\end{equation*}
Applying Lemma \ref{lemA3} with the symmetrization inequality \citep[][Lemma 2.3.1]{VW96} and the contraction principle \citep[][Theorem 4.12]{LT91}, for all $\epsilon \geq C_{\nu} n^{-\nu/(2\nu+1)}$, we have $\mathrm{E}[ Z_{j} ] \leq C_{q,\nu} \epsilon^{(2\nu-1)/(2\nu)}$.
Therefore, taking $t=2 \log (s \vee n)$, we have, 
with probability at least $1-(s \vee n)^{-2}$, 
\begin{equation*}
Z_{j} \leq  C_{q,\nu} \max \{  \epsilon^{(2\nu -1)/(2\nu)}, \epsilon^{-1/(2\nu)} \sqrt{\log (s \vee n)/n}, \epsilon^{-1/\nu} \log (s \vee n)/n \}. 
\end{equation*}

By the union bound, the above inequality simultaneously holds for all $j \in T$ with probability at least $1-(s \vee n)^{-1}$. 
The desired conclusion now follows from the additional restriction that $C_{q,\nu} \epsilon^{-1/(2\nu)} \max \{ \epsilon,  \sqrt{\log (s \vee n)/n} \} \leq 0.5$. 
\end{proof}

\begin{lemma}
\label{lemA5}
Let $\mathcal{P}_{\nu-1}$ denote the set of all polynomial functions on $[0,1]$ of degree $\nu-1$. Assume conditions (C1), (C3) and (C5). Then, with probability approaching one, 
$\| h_{j} \|^{2}_{2,n} \leq 1.5 \| h_{j} \|^{2}_{2}$ for all $h_{j} \in \mathcal{P}_{\nu-1}$ and $1 \leq j \leq d$. 
\end{lemma}

\begin{proof}
As in the previous proof, $C_{q,\nu}$ denotes some positive constant depending only on $c_{q}$ and $\nu$. Its value may change from line to line.
By normalization, it suffices to show that 
$\max_{1 \leq j \leq d} \sup_{h_{j} \in \mathcal{H}_{j}}  | \| h_{j} \|_{2,n}^{2} - 1 | \stackrel{p}{\to} 0$, where $\mathcal{H}_{j} := \{ h_{j} : h_{j} \in \mathcal{P}_{\nu-1}, \| h_{j} \|_{2}  = 1 \}$. Pick any $1 \leq j \leq d$.
By condition (C3) (ii) and Lemma \ref{lemA1}, $\| h_{j} \|_{\infty} \leq C_{q,\nu} \| h_{j} \|_{2} = C_{q,\nu}$ for all $h_{j} \in \mathcal{H}_{j}$, 
and $\mathrm{E}[ h_{j}(z_{1j})^{4} ] \leq \|  h_{j} \|_{\infty}^{2} \| h_{j} \|_{2}^{2} \leq C_{q,\nu}$. Therefore, by Massart's form of Talagrand's inequality, 
for all $t > 0$, with probability at least $1-e^{-t}$,
\begin{equation*}
Z_{j} \leq 2\mathrm{E}[ Z_{j} ] + C_{q,\nu} \sqrt{t/n} + C_{q,\nu} t/n,
\end{equation*}
where $Z_{j} := \sup_{h_{j} \in \mathcal{H}_{j}}  | \| h_{j} \|_{2,n}^{2} - 1 |$. We wish to evaluate $\mathrm{E}[Z_{j}]$. Let $\sigma_{1},\dots,\sigma_{n}$ be independent Rademacher random variables independent of $\bm{z}_{1},\dots,\bm{z}_{n}$. By the symmetrization inequality and the contraction principle, 
\begin{equation*}
\mathrm{E}[ Z_{j} ] \leq C_{q,\nu} \mathrm{E}\left[ \sup_{h_{j} \in \mathcal{H}_{j}} \left | \frac{1}{n} \sum_{i=1}^{n} \sigma_{i} h_{j}(z_{ij}) \right | \right].
\end{equation*}
Arguing as in \citet[][p.3813]{MVB09} (or using a standard entropy integral argument), it is shown that 
\begin{equation*}
\mathrm{E}\left[ \sup_{h_{j} \in \mathcal{H}_{j}} \left | \frac{1}{n} \sum_{i=1}^{n} \sigma_{i} h_{j}(z_{ij}) \right | \right] \leq \frac{C_{q,\nu}}{\sqrt{n}}.
\end{equation*}
Taking $t=2\log d$, we have, with probability at least $1-d^{-2}$, 
\begin{equation*}
Z_{j} \leq C_{q,\nu} \sqrt{\frac{\log d}{n}}. 
\end{equation*}  
(Recall that $\log d/n \to 0$.) 
 
By the union bound, the above inequality simultaneously holds for all $1 \leq j \leq d$ with probability at least $1-d^{-1}$. 
Recalling that $d \to \infty$, we obtain the desired conclusion. 
\end{proof}

\subsection{Proof of Lemma 6.1}

{\bf Parts (i) and (ii)}: 
We first point out that (ii) follows from (i). Suppose that (i) is true for the case that $u_{i}$ are independent Rademacher random variables independent of $\bm{z}_{1},\dots,\bm{z}_{n}$. 
Let $\sigma_{1},\dots,\sigma_{i}$ denote independent Rademacher random variables independent of $\bm{z}_{1},\dots,\bm{z}_{n}$. By the symmetrization inequality for probabilities \citep[][Lemma 2.13]{VW96}, for all $t \geq \sqrt{8/n}$, 
{\footnotesize
\begin{equation*}
\mathrm{P} \left \{ \sup_{g_{j} \in \mathcal{G}} \frac{|n^{-1} \sum_{i=1}^{n} \{ g_{j}(z_{ij}) - \mu_{g_{j}} \}|}{\sqrt{\| g_{j} \|_{2}^{2} + \epsilon^{2} I(g_{j})^{2}}} > t \right \} \leq 4\mathrm{P} \left \{ \sup_{g_{j} \in \mathcal{G}} \frac{|n^{-1} \sum_{i=1}^{n} \sigma_{i} g_{j}(z_{ij}) |}{\sqrt{\| g_{j} \|_{2}^{2} + \epsilon^{2} I(g_{j})^{2}}} > \frac{t}{4} \right \}.
\end{equation*}
}
Thus, by (i), the right side goes to zero with $t=4 \max \{ \epsilon, C_{1} \sqrt{\log (s \vee n)/n} \}$.

In what follows, we wish to show (i). 
Take $C_{\nu}$ as in Lemma \ref{lemA3} and let $\epsilon = \epsilon_{n} \to 0$ be any sequence such that $\epsilon \geq C_{\nu} n^{-\nu/(2\nu+1)}$. 
Define the events  
\begin{align*}
\Omega_{7} &:= \{ \| g_{j} \|_{2,n}^{2} \leq 1.5 \| g_{j} \|_{2}^{2} + 0.5 \epsilon^{2} I(g_{j})^{2}, \ \forall g_{j} \in \mathcal{G}, \ \forall j \in T \}, \\
\Omega_{8} &:= \{ \| h_{j} \|_{2,n}^{2} \leq 1.5 \| h_{j} \|_{2}^{2}, \ \forall h_{j} \in \mathcal{P}_{\nu-1}, \ \forall j \in T \},  
\end{align*}
where $\mathcal{P}_{\nu-1}$ denotes the set of all polynomial functions on $[0,1]$ of degree $\nu-1$. 

We first consider case (a) in condition (C2). By normalization, 
it suffices to consider the case that $| u_{1} | \leq 1$ almost surely. 
Pick any $j \in T$.
Consider the function
\begin{align*}
F_{j}(\bm{u},\bm{z}_{\cdot j}) &:= \sup_{g_{j} \in \mathcal{G}} \frac{| n^{-1}\sum_{i=1}^{n} u_{i} g_{j}(z_{ij}) |}{\sqrt{\| g_{j} \|_{2}^{2} + \epsilon^{2} I(g_{j})^{2}}},  \\
&= \sup_{g_{j} \in \mathcal{G}} \frac{ n^{-1}\sum_{i=1}^{n} u_{i} g_{j}(z_{ij}) }{\sqrt{\| g_{j} \|_{2}^{2} + \epsilon^{2} I(g_{j})^{2}}}, \ \bm{u} = (u_{1},\dots,u_{n})'\ \bm{z}_{\cdot j} = (z_{1j},\dots,z_{nj})',
\end{align*}
where the second inequality is due to the fact that $\mathcal{G}$ is symmetric, i.e., if $g_{j} \in \mathcal{G}$ then $-g_{j} \in \mathcal{G}$. 
Given $\bm{z}_{1},\dots,\bm{z}_{n}$, the map $\bm{u} \mapsto F_{j}(\bm{u},\bm{z}_{\cdot j})$ is Lipschitz continuous with Lipschitz constant bounded by 
\begin{equation*}
\sup_{g_{j} \in \mathcal{G}} \frac{n^{-1/2} \| g_{j} \|_{2,n}}{\sqrt{\| g_{j} \|_{2}^{2} + \epsilon^{2} I(g_{j})^{2}}},
\end{equation*}
which is, on the event $\Omega_{7}$, bounded by $Cn^{-1/2}$. Therefore, by Corollary 4.8 of \cite{L01}, on the event $\Omega_{7}$, 
\begin{equation*}
\mathrm{P}_{u} \{ F_{j}(\bm{u},\bm{z}_{\cdot j}) \geq \mathrm{E}_{u}[F_{j}(\bm{u},\bm{z}_{\cdot j})] +  Ctn^{-1/2}  \} \leq c_{1} e^{-c_{2}t^{2}},
\end{equation*}
where $\mathrm{P}_{u}$ ($\mathrm{E}_{u}$) denotes the conditional probability (expectation, respectively) with respect to $u_{1},\dots,u_{n}$ given $\bm{z}_{1},\dots,\bm{z}_{n}$, and $c_{1} > 0$ and $c_{2} > 0$ are universal constants.

By Lemma \ref{lemA2}, invoke that $g_{j} \in \mathcal{G}$ can be written as $g_{j} = g_{j}^{[1]} + g_{j}^{[2]}$ 
such that (i) $\| g_{j}^{[1]} \|_{\infty} \leq \const \times I(g^{[1]}_{j})$ (the constant is independent of $g_{j}$); 
(ii) $g_{j}^{[2]} \in \mathcal{P}_{\nu-1}$;  and (iii) $\mathrm{E}[ g_{j}^{[1]}(z_{1j}) g_{j}^{[2]}(z_{1j}) ] = 0$.
Observe that on the event $\Omega_{7} \cap \Omega_{8}$, 
{\footnotesize
\begin{align*}
\frac{| n^{-1}\sum_{i=1}^{n} u_{i} g_{j}(z_{ij}) |}{\sqrt{\| g_{j} \|_{2}^{2} + \epsilon^{2} I(g_{j})^{2}}} &\leq 
\frac{| n^{-1}\sum_{i=1}^{n} u_{i} g_{j}^{[1]}(z_{ij}) |}{\sqrt{\| g^{[1]}_{j} \|_{2}^{2} + \epsilon^{2} I(g^{[1]}_{j})^{2}}} + \frac{| n^{-1}\sum_{i=1}^{n} u_{i} g^{[2]}_{j}(z_{ij}) |}{\| g^{[2]}_{j} \|_{2}} \\
&\leq C \left \{ \frac{| n^{-1}\sum_{i=1}^{n} u_{i} g_{j}^{[1]}(z_{ij}) |}{\sqrt{\| g^{[1]}_{j} \|_{2,n}^{2} + \epsilon^{2} I(g^{[1]}_{j})^{2}}} + \frac{| n^{-1}\sum_{i=1}^{n} u_{i} g^{[2]}_{j}(z_{ij}) |}{\| g^{[2]}_{j} \|_{2,n}} \right \}.
\end{align*}
}
By Lemma \ref{lemA3} (see also Addendum \ref{add1}) and the fact that $u_{i}$'s can be replaced by independent Rademacher variables by the contraction principle, the conditional expectation (given $\bm{z}_{1},\dots,\bm{z}_{n}$) of the first term inside the brace is bounded by $C_{\nu} \epsilon$. 
On the other hand, by \citet[][p.3813]{MVB09}, the conditional expectation of the second term inside the brace is bounded by a constant times $n^{-1/2}$ where the constant depends only on $\nu$. So there exists a constant $C_{\nu}'$ depending only on $\nu$ such that  $\mathrm{E}_{u}[F_{j}(\bm{u},\bm{z}_{\cdot j})]  \leq C'_{\nu} \epsilon$ on the event $\Omega_{7} \cap \Omega_{8}$. Therefore, 
on the event $\Omega_{7} \cap \Omega_{8}$, 
\begin{equation*}
\mathrm{P}_{u} \{ F_{j}(\bm{u},\bm{z}_{\cdot j}) \geq C'_{\nu} \epsilon + C \sqrt{\log (s \vee n)/n}  \} \leq c_{1}(s \vee n)^{-2},
\end{equation*}
where we have taken $t = \sqrt{2c_{2}^{-1}\log (s \vee n)}$.

We now move $j$. By the union bound, 
\begin{align*}
&\mathrm{P} \{ \max_{j \in T} F_{j}(\bm{u},\bm{z}_{\cdot j}) \geq C'_{\nu} \epsilon + C \sqrt{\log (s \vee n)/n} \} \\
&\leq \mathrm{P} \left [ \{ \max_{j \in T} F_{j}(\bm{u},\bm{z}_{\cdot j}) \geq \epsilon \} \cap \Omega_{7} \cap \Omega_{8} \right ] + \mathrm{P}(\Omega_{7}^{c}) + \mathrm{P}(\Omega_{8}^{c}) \\
&\leq c_{1}(s \vee n)^{-1} + \mathrm{P}(\Omega_{7}^{c}) + \mathrm{P}(\Omega_{8}^{c}).
\end{align*}
By Lemmas \ref{lemA3} and \ref{lemA4}, $\mathrm{P}(\Omega_{7}^{c}) + \mathrm{P}(\Omega_{8}^{c}) \to 0$. Therefore, we obtain (i) for the case that $| u_{1} | \leq 1$ almost surely. 

We next consider case (b) in condition (C2). Recall that $u_{1} | \bm{z}_{1} \sim N(0, \sigma_{u}(\bm{z}_{1})^{2})$ and $\sigma_{u}(\bm{z}_{1}) \leq \sigma_{u}$ almost surely. Put $\tilde{u}_{i} := u_{i}/\sigma_{u}(\bm{z}_{i})$. Then, $\tilde{u}_{1},\dots,\tilde{u}_{n}$ are independent standard normal random variables independent of $\bm{z}_{1},\dots,\bm{z}_{n}$. Consider now the function
\begin{equation*}
F_{j}(\tilde{\bm{u}},\bm{z}_{1}^{n}) := \sup_{g_{j} \in \mathcal{G}} \frac{| n^{-1}\sum_{i=1}^{n} \tilde{u}_{i} \sigma_{u}(\bm{z}_{i})g_{j}(z_{ij}) |}{\sqrt{\| g_{j} \|_{2}^{2} + \epsilon^{2} I(g_{j})^{2}}}, \ \tilde{\bm{u}} = (\tilde{u}_{1},\dots,\tilde{u}_{n})', \ \bm{z}_{1}^{n} = \{ \bm{z}_{1},\dots,\bm{z}_{n} \}.
\end{equation*}
Given $\bm{z}_{1},\dots,\bm{z}_{n}$, the map $\tilde{\bm{u}} \mapsto F_{j}(\tilde{\bm{u}},\bm{z}_{1}^{n})$ is Lipschitz continuous with Lipschitz constant bounded by 
\begin{equation*}
\sigma_{u} \sup_{g_{j} \in \mathcal{G}} \frac{n^{-1/2} \|  g_{j} \|_{2,n}}{\sqrt{\| g_{j} \|_{2}^{2} + \epsilon^{2} I(g_{j})^{2}}},
\end{equation*}
which is, on the event $\Omega_{7}$, bounded by $C\sigma_{u} n^{-1/2}$. Therefore, by Theorem 7.1 of \cite{L01}, on the event $\Omega_{7}$, 
\begin{equation*}
\mathrm{P}_{\tilde{u}} \{ F_{j}(\tilde{\bm{u}},\bm{z}_{1}^{n}) \geq \mathrm{E}_{\tilde{u}}[F_{j}(\tilde{\bm{u}},\bm{z}_{1}^{n})] +  C\sigma_{u} tn^{-1/2}  \} \leq  e^{-t^{2}/2}.
\end{equation*}
By the contraction principle for Gaussian processes \citep[][Corollary 3.17]{LT91},
\begin{equation*}
\mathrm{E}_{\tilde{u}}[F_{j}(\tilde{\bm{u}},\bm{z}_{1}^{n})]  
\leq C \sigma_{u}\mathrm{E}_{\tilde{u}} \left [ \sup_{g_{j} \in \mathcal{G}} \frac{| n^{-1}\sum_{i=1}^{n} \tilde{u}_{i} g_{j}(z_{ij}) |}{\sqrt{\| g_{j} \|_{2}^{2} + \epsilon^{2} I(g_{j})^{2}}} \right ].
\end{equation*}
The rest of the procedure is the same as the previous one. 

{\bf Part (iii)}: The result basically follows from the same argument as the proof of \citet[][Theorem 6]{MVB09}.
For the sake of completeness, we provide an outline of the proof. In what follows, $C_{q,\nu}$ denotes some constant depending only on $c_{q}$ and $\nu$. Its value may change from line to line. 

Let $\sigma_{1},\dots,\sigma_{n}$ denote independent Rademacher random variables independent of $\bm{z}_{1},\dots,\bm{z}_{n}$.  
Define 
\begin{align*}
&Z := \sup_{g = \sum_{j=1}^{d} g_{j}, g_{j} \in \mathcal{G}} \frac{| \| g \|^{2}_{2,n} - \| g \|^{2}_{2} |}{ [\sum_{j=1}^{d} \sqrt{ \| g_{j} \|_{2}^{2} + \epsilon^{2} I(g_{j})^{2} }]^{2}}, \\
&\tilde{Z} := \sup_{g = \sum_{j=1}^{d} g_{j}, g_{j} \in \mathcal{G}} \frac{| n^{-1} \sum_{i=1}^{n} \sigma_{i} g(\bm{z}_{i}) |}{ \sum_{j=1}^{d} \sqrt{\| g_{j} \|_{2}^{2} + \epsilon^{2} I(g_{j})^{2} }}.
\end{align*}
By Lemma \ref{lemA1}, for $g = \sum_{j=1}^{d} g_{j}, g_{j} \in \mathcal{G}$,  
\begin{equation*}
\| g \|_{\infty} \leq \sum_{j=1}^{d} \| g_{j} \|_{\infty} \leq C_{q,\nu} e^{-1/(2\nu)} \sum_{j=1}^{d} \sqrt{\| g_{j} \|_{2}^{2} + \epsilon^{2} I(g_{j})^{2}},
\end{equation*}
and 
\begin{equation*}
\mathrm{E}[ g(\bm{z}_{1})^{4}] \leq \| g \|_{\infty}^{2} \| g \|_{2}^{2} \leq C_{q,\nu} e^{-1/\nu} \left [ \sum_{j=1}^{d} \sqrt{\| g_{j} \|_{2}^{2} + \epsilon^{2} I(g_{j})^{2}} \right ]^{4},
\end{equation*}
where we have use the inequality that $\| g \|_{2} \leq \sum_{j=1}^{d} \| g_{j} \|_{2} \leq \sum_{j=1}^{d} \sqrt{ \| g_{j} \|_{2}^{2} + \epsilon^{2} I(g_{j})^{2}}$. 
Thus, by Massart's (2000) form of Talagrand's (1996) inequality, for all $t > 0$, with probability at least $1-e^{-t}$, 
\begin{equation*}
Z \leq 2 \mathrm{E}[ Z ] + C_{q,\nu} \sqrt{\epsilon^{-1/\nu}t/n} + C_{q,\nu} \epsilon^{-1/\nu} t/n.
\end{equation*}
We wish to evaluate $\mathrm{E}[ Z ]$.
Using the symmetrization inequality and the contraction principle, we have
\begin{equation*}
\mathrm{E}[ Z ] \leq C_{q,\nu} \epsilon^{-1/(2\nu)} \mathrm{E}[\tilde{Z}].
\end{equation*}
By a standard calculation, 
\begin{equation*}
\mathrm{E}[ \tilde{Z} ] \leq \mathrm{E}\left[ \max_{1 \leq j \leq d} \sup_{g_{j} \in \mathcal{G}} \frac{|n^{-1}\sum_{i=1}^{n}\sigma_{i}g_{j}(z_{ij})|}{\sqrt{\| g_{j} \|_{2}^{2} + \epsilon^{2} I(g_{j})^{2}}} \right ].
\end{equation*}
Recall that $\delta = \max \{ n^{-\nu/(2\nu+1)}, \sqrt{\log d/n} \}$. By lemma 13 of \cite{MVB09} and Lemma \ref{lemA3}, we have
{\footnotesize
\begin{align*}
&\mathrm{E}\left[ \max_{1 \leq j \leq d} \sup_{g_{j} \in \mathcal{G}} \frac{|n^{-1}\sum_{i=1}^{n}\sigma_{i}g_{j}(z_{ij})|}{\sqrt{\| g_{j} \|_{2}^{2} + \epsilon^{2} I(g_{j})^{2}}} \right ] \\
&\leq 
4 \max_{1 \leq j \leq d} \mathrm{E}\left[ \sup_{g_{j} \in \mathcal{G}} \frac{|n^{-1}\sum_{i=1}^{n}\sigma_{i}g_{j}(z_{ij})|}{\sqrt{\| g_{j} \|_{2}^{2} + \epsilon^{2} I(g_{j})^{2}}} \right ] + \frac{2C_{q,\nu} \epsilon^{-1/(2\nu)} (1+\log d)}{3n} + \sqrt{\frac{4(1+\log d)}{n}} \\
&\leq C_{q,\nu}\max \left \{ \epsilon,  \frac{\epsilon^{-1/(2\nu)}\log d}{n}, \sqrt{\frac{\log d}{n}} \right \} \\
&\leq C_{q,\nu} \max \{ \epsilon, \delta \}.
\end{align*}
}
The last inequality is because 
\begin{equation*}
\frac{\epsilon^{-1/(2\nu)}\log d}{n} \lesssim \sqrt{\frac{\log d}{n^{2\nu/(2\nu+1)}}} \times \sqrt{\frac{\log d}{n}} = o(1) \sqrt{\frac{\log d}{n}}.
\end{equation*}
Thus, with probability at least $1-e^{-t}$, 
\begin{equation*}
Z \leq C_{q,\nu} \max \{ \epsilon^{(2\nu-1)/(2\nu)},  \epsilon^{-1/(2\nu)} \delta, \sqrt{\epsilon^{-1/\nu}t/n}, \epsilon^{-1/\nu} t/n \}.
\end{equation*}
Letting $t \to \infty$ sufficiently slowly (such that $\sqrt{t/n} \lesssim \epsilon$), we have, 
with probability approaching one, $Z \leq C_{q,\nu} \epsilon^{-1/(2\nu)} \max \{ \epsilon, \delta \}$, which implies the desired conclusion. 

\qed

\section{Proofs for Section 4}

\subsection{Proofs of Propositions 4.1 and 4.2}

We first point out that since $\sum_{i=1}^{n} \tilde{\bm{x}}_{i} = \bm{0}$, by a standard argument, 
we may assume that $c^{*} = \mathrm{E}[ y_{1} ] = 0$ for the analysis of $\hat{g}$.

\begin{proof}[Proof of Proposition 4.1]

The proof is a direct adaptation of that of \citet[][Theorem 6.1]{BRT09}, so we omit the detail here. 
\end{proof}

For a subset $T \subset \{ 1,\dots, d\}$, let $\bm{x}_{iG_{T}}$ denote the $m |T|  \times 1$ vector stacked by $\bm{x}_{iG_{j}}, j \in T$. 

\begin{proof}[Proof of Proposition 4.2]
Recall that $\hat{T}^{0} := \{ j \in \{ 1,\dots, d \} : \| \hat{g}_{j} \|_{2,n} > 0 \}$. Note that on the event $\Omega_{0}$, $\hat{T}^{0} = \{ j \in \{ 1,\dots, d \} : \| \hat{\bm{\beta}}_{G_{j}} \|_{E} \neq 0 \}$. 
By the Karush-Kuhn-Tucker condition, on the event $\Omega_{0}$, 
\begin{equation*}
\frac{\sqrt{m}\lambda_{1}}{n} \frac{\hat{\bm{\Sigma}}^{1/2}_{j}\hat{\bm{\beta}}_{G_{j}}}{\| \hat{\bm{\Sigma}}_{j}^{1/2} \hat{\bm{\beta}}_{G_{j}} \|_{E}} = \frac{1}{n} \sum_{i=1}^{n} \check{\bm{x}}_{iG_{j}}(y_{i} - \tilde{\bm{x}}_{i}'\hat{\bm{\beta}}), \ \forall j \in \hat{T}^{0}, 
\end{equation*}
which implies that on the event $\{ \lambda_{1} \geq 2  \Lambda \} \cap \Omega_{0}$, 
\begin{align*}
\frac{\sqrt{m}\lambda_{1}}{n} \sqrt{\hat{s}} &\leq \left \| \frac{1}{n} \sum_{i=1}^{n} \check{\bm{x}}_{iG_{\hat{T}^{0}}}(y_{i} - \tilde{\bm{x}}_{i}'\hat{\bm{\beta}}) \right \|_{E} \\
&\leq \left \| \frac{1}{n} \sum_{i=1}^{n}\check{\bm{x}}_{iG_{\hat{T}^{0}}}(u_{i} + r_{i}) \right \|_{E} \quad (r_{i} := g^{*}(\bm{z}_{i}) - \hat{g}(\bm{z}_{i}) ) \\
&\leq \left \| \frac{1}{n} \sum_{i=1}^{n} u_{i} \check{\bm{x}}_{iG_{\hat{T}^{0}}} \right \|_{E} + \left \| \frac{1}{n} \sum_{i=1}^{n} r_{i} \check{\bm{x}}_{iG_{\hat{T}^{0}}} \right \|_{E} \\
&\leq \frac{\sqrt{m}\lambda_{1}}{2n} \sqrt{\hat{s}} + 1.5 \left \| \frac{1}{n} \sum_{i=1}^{n} r_{i} \tilde{\bm{x}}_{iG_{\hat{T}^{0}}} \right \|_{E}.
\end{align*}
Applying the Cauchy-Schwarz inequality to the last line, we obtain 
\begin{equation*}
\frac{\sqrt{m}\lambda_{1}}{n} \sqrt{\hat{s}} \leq 3 \left \| \frac{1}{n} \sum_{i=1}^{n} r_{i} \tilde{\bm{x}}_{iG_{\hat{T}^{0}}} \right \|_{E} \leq 3  \| g^{*} - \hat{g} \|_{2,n}  \hat{\phi}_{\max}(\hat{s}).
\end{equation*} 
Thus, on the event $\{ \lambda_{1} \geq 2  \Lambda \vee 0 \} \cap \Omega_{0}$, 
\begin{equation*}
\hat{s} \leq \hat{C}^{2} \hat{\phi}_{\max}(\hat{s})^{2}s^{*}.
\end{equation*}
We wish to show that $\hat{\phi}_{\max}(\hat{s})^{2} \leq \min_{s \in \mathcal{S}} \hat{\phi}_{\max}(s)^{2}$. 
Because the map $s \mapsto \hat{\phi}_{\max} (s)^{2}$ is non-decreasing, it suffices to show that $\hat{s} \leq s$ for any $s \in \mathcal{S}$. 
Pick any $s \in \mathcal{S}$. Suppose on the contrary that $\hat{s} > s$. Then, 
{\footnotesize
\begin{align*}
\hat{s} \leq \hat{C}^{2} \hat{\phi}_{\max}(\hat{s})^{2}s^{*} = \hat{C}^{2} \hat{\phi}_{\max}(s \cdot (\hat{s}/s))^{2}s^{*} \leq \hat{C}^{2} \lceil \hat{s}/s \rceil \hat{\phi}_{\max}(s)^{2}s^{*} \leq 2\hat{C}^{2} (\hat{s}/s) \hat{\phi}_{\max}(s)^{2}s^{*} < \hat{s},
\end{align*}
}
a contradiction (we have used the property $\hat{\phi}_{\max}(ls)^{2} \leq \lceil l \rceil \hat{\phi}_{\max}(s)^{2}$ for $l > 0$, which can be shown is a similar way as the proof of \citet[][Lemma 8]{BC09}). Therefore, we obtain the desired conclusion. 
\end{proof}

\subsection{Proof of Theorem 4.2}

Notation: In condition (C5), let $K$ denote a fixed constant such that 
\begin{equation*}
\sup_{z \in [0,1]} \| (\psi_{1}(z),\dots,\psi_{m}(z))' \|_{E} \leq K m^{1/2}.
\end{equation*} 

Theorem 4.2 follows from Propositions 4.1 and 4.2 together with Lemmas \ref{lem2}-\ref{lem6} below. In what follows, we always assume conditions (C1)-(C5) and (C7).

\begin{lemma}
\label{lem2}
Assume that $m \log d/n \to 0$. Then, $\mathrm{P}(\Omega_{0}) \to 1$.
\end{lemma}

\begin{proof}[Proof of Lemma \ref{lem2}]
Without loss of generality, we may assume that $\mathrm{E}[\bm{x}_{1}] = \bm{0}$. It suffices to show that $\max_{1 \leq j \leq d} \| \hat{\bm{\Sigma}}_{j} - \bm{I}_{m} \| \stackrel{p}{\to} 0$. Because $\hat{\bm{\Sigma}}_{j} = \hat{\bm{\Sigma}}_{0j} - \bar{\bm{x}}_{G_{j}}\bar{\bm{x}}_{G_{j}}' \ (\hat{\bm{\Sigma}}_{0j} := n^{-1} \sum_{i=1}^{n} \bm{x}_{iG_{j}} \bm{x}_{iG_{j}}')$, it suffices to show that 
\begin{equation*}
\max_{1 \leq j \leq d} \| \bar{\bm{x}}_{G_{j}} \|_{E} \stackrel{p}{\to} 0, \max_{1 \leq j \leq d} \| \hat{\bm{\Sigma}}_{0j} - \bm{I}_{m} \| \stackrel{p}{\to} 0.
\end{equation*}
The second assertion follows from Lemma 3.2 of \cite{K11}. We wish to show the first assertion. Pick any $j \in \{ 1,\dots,d \}$. By Corollary 4.5 of \cite{L01}, 
for all $t > 0$, with probability at least $1-e^{-t^{2}/2}$, we have 
\begin{equation*}
\| \bar{\bm{x}}_{G_{j}} \|_{E} \leq (1+tK)\sqrt{m/n},
\end{equation*}
where we have used the fact that $\mathrm{E}[\| \bar{\bm{x}}_{G_{j}} \|_{E}] \leq \sqrt{n^{-2} \sum_{i=1}^{n} \mathrm{E}[\| \bm{x}_{iG_{j}} \|_{E}^{2}]} =\sqrt{m/n}$ and $| \bm{\alpha}'\bm{x}_{1G_{j}} | \leq K \sqrt{m}$ (since we are considering the one-sided deviation inequality, $2$ in front of the exponential term in Corollary 4.5 of \cite{L01} can be replaced by $1$). 
By the union bound, the above inequality simultaneously holds for all $1 \leq j \leq d$ with probability at least $1-de^{-t^{2}/2}$.
Taking $t=2\sqrt{\log d}$, we obtain the desired result.  
\end{proof}

\begin{lemma}
\label{lem3}
There exists a positive constant $A_{1,u}$ depending only on the distribution of $u_{1}$ such that for any $\lambda_{1}$ satisfying 
\begin{equation*}
\lambda_{1} \geq A_{1,u} \sqrt{n} \left ( 1 +  \sqrt{\frac{\log d}{m}} \right ),
\end{equation*}
we have $\mathrm{P}\{ \lambda_{1} \geq 2  \Lambda \} \to 1$. 
\end{lemma}

\begin{proof}[Proof of Lemma \ref{lem3}]
This follows from deviation inequalities in product spaces. We first consider case (a) in condition (C2). By normalization, 
it suffices to consider the case that $| u_{1} | \leq 1$ almost surely. 
Pick any $1 \leq j \leq d$. Consider the function
\begin{equation*}
F_{j}(\bm{u},\bm{z}_{\cdot j}) := \left \| \sum_{i=1}^{n} u_{i} \check{\bm{x}}_{iG_{j}} / \sqrt{m} \right \|_{E}, \ \text{where} \ \bm{u}=(u_{1},\dots,u_{n})', \ \bm{z}_{\cdot j} = (z_{1j}, \dots, z_{nj})'.
\end{equation*}
(Recall that $\bm{x}_{iG_{j}}$ is generated by $z_{ij}$.) Given $\bm{z}_{1},\dots,\bm{z}_{n}$, the map $\bm{u} \mapsto F_{j}(\bm{u},\bm{z}_{\cdot j})$ is  Lipschitz continuous with Lipschitz constant bounded by $\sqrt{n/m}$ (invoke that the maximum eigenvalue of $n^{-1} \sum_{i=1}^{n} \check{\bm{x}}_{iG_{j}}\check{\bm{x}}_{iG_{j}}'$ is $1$). Thus, by Corollary 4.8 of \cite{L01}, for all $t > 0$, 
\begin{equation*}
\mathrm{P}_{u}\{ F_{j}(\bm{u},\bm{z}_{\cdot j}) \geq \mathrm{E}_{u}[F_{j}(\bm{u},\bm{z}_{\cdot j})] + Ct \sqrt{n/m} \} \leq c_{1}e^{-c_{2} t^{2}},
\end{equation*}
where $\mathrm{P}_{u}$  $(\mathrm{E}_{u})$ denotes the conditional probability (expectation, respectively) of $u_{1},\dots,u_{n}$ given $\bm{z}_{1},\dots,\bm{z}_{n}$, and $c_{1} > 0$ and $c_{2} > 0$ are universal constants.
A direct calculation shows that $\mathrm{E}_{u}[ F_{j}(\bm{u},\bm{z}_{\cdot j})] \leq \sqrt{n}$ (invoke that the trace of the matrix $n^{-1} \sum_{i=1}^{n} \check{\bm{x}}_{iG_{j}} \check{\bm{x}}_{iG_{j}}'$ is bounded by $m$). Therefore, taking $t= \sqrt{2c_{2}^{-1}\log d}$, we have, with probability at least $1-c_{1}d^{-2}$, 
\begin{equation*}
F_{j}(\bm{u},\bm{z}_{\cdot j}) \leq \sqrt{n} + C \sqrt{n\log d /m}.
\end{equation*}
By the union bound, the above inequality simultaneously holds for all $1 \leq j \leq d$ with probability at least $1-c_{1}d^{-1}$. 
Recalling that $d \to \infty$, we obtain the desired conclusion for the case that $| u_{1} | \leq 1$ almost surely. 

In case (b) in Condition (C2), we use Theorem 7.1 of \cite{L01}. Recall that $u_{1} | \bm{z}_{1} \sim N(0,\sigma_{u}(\bm{z}_{1})^{2})$ and $\sigma_{u}(\bm{z}_{1}) \leq \sigma_{u}$ almost surely. 
Put $\tilde{u}_{i} := u_{i}/\sigma_{u}(\bm{z}_{i})$. Then, $\tilde{u}_{1},\dots,\tilde{u}_{n}$ are independent standard normal random variables independent of 
$\bm{z}_{1},\dots,\bm{z}_{n}$. Define 
\begin{equation*}
F_{j}(\tilde{\bm{u}},\bm{z}_{1}^{n}) := \left \| \sum_{i=1}^{n} \tilde{u}_{i} \sigma_{u}(\bm{z}_{i}) \check{\bm{x}}_{iG_{j}} / \sqrt{m} \right \|_{E}, \ \tilde{\bm{u}}=(\tilde{u}_{1},\dots,\tilde{u}_{n})', \ \bm{z}_{1}^{n} = \{ \bm{z}_{1},\dots,\bm{z}_{n} \}'.
\end{equation*}
Given $\bm{z}_{1},\dots,\bm{z}_{n}$, the map $\tilde{\bm{u}} \mapsto F_{j}(\tilde{\bm{u}},\bm{z}_{1}^{n})$ is Lipschitz continuous with Lipschitz constant bounded by $\sigma_{u}\sqrt{n/m}$. Thus, by Theorem 7.1 of \cite{L01}, for all $t > 0$, 
\begin{equation*}
\mathrm{P}_{\tilde{u}} \{ F_{j}(\tilde{\bm{u}},\bm{z}_{1}^{n}) \geq \mathrm{E}_{\tilde{u}}[F_{j}(\tilde{\bm{u}},\bm{z}_{1}^{n})] +  \sigma_{u} t\sqrt{n/m} \} \leq e^{-t^{2}/2},
\end{equation*}
where, by a direct calculation, $\mathrm{E}_{\tilde{u}}[F_{j}(\tilde{\bm{u}},\bm{z}_{1}^{n})] \leq \sigma_{u} \sqrt{n}$. The rest of the procedure is exactly the same as the previous one. 
\end{proof}

\begin{lemma}
\label{lem4}
Let $\phi_{\max}(s)$ denote the $s$-th group sparse maximum eigenvalue of $\bm{\Sigma}^{1/2}$: $\phi_{\max}(s) := \sup_{| T | \leq s,\bm{\alpha} \in \mathbb{S}_{T}^{dm-1}} \| \bm{\Sigma}^{1/2}\bm{\alpha} \|_{E}$. 
Assume that $s,m,d$ and $n$ obey the growth condition 
\begin{equation}
\frac{s^{2}m\log (d \vee n)}{n} \to 0.
\label{growth}
\end{equation}
Then, $\hat{\phi}_{\max}(s) \lesssim_{p} 1$ provided that $\phi_{\max}(s) \lesssim 1$. 
\end{lemma}

Recall that for a subset $T \subset \{ 1,\dots, d\}$, $\bm{x}_{iG_{T}}$ denotes the $m |T|  \times 1$ vector stacked by $\bm{x}_{iG_{j}}, j \in T$.  We use the notation: $\hat{\bm{\Sigma}}_{T} := n^{-1} \sum_{i=1}^{n} \tilde{\bm{x}}_{iG_{T}}\tilde{\bm{x}}_{iG_{T}}', \hat{\bm{\Sigma}}_{0T} := n^{-1} \sum_{i=1}^{n} \bm{x}_{iG_{T}}\bm{x}_{iG_{T}}'$ and $\bm{\Sigma}_{T}:=\mathrm{E}[\bm{x}_{1G_{T}}\bm{x}_{1G_{T}}']$.

\begin{proof}[Proof of Lemma \ref{lem4}]

Without loss of generality, we may assume that $\mathrm{E}[\bm{x}_{1}]=\bm{0}$.
Pick any subset $T$ of $\{ 1,\dots,d \}$ such that $| T | = s$. We wish to evaluate a tail probability of $\| \hat{\bm{\Sigma}}_{T} - \bm{\Sigma}_{T} \|$. 
Since $\| \hat{\bm{\Sigma}}_{T} - \bm{\Sigma}_{T} \| \leq \| \hat{\bm{\Sigma}}_{0T} - \bm{\Sigma}_{T} \| + \| \bar{\bm{x}}_{G_{T}} \|_{E}^{2}$, we separately 
evaluate the right two terms. 

We first evaluate the term $\| \hat{\bm{\Sigma}}_{0T}-\bm{\Sigma}_{T} \|$. Invoke the expression
\begin{equation*}
\| \hat{\bm{\Sigma}}_{0T} - \bm{\Sigma}_{T} \| = \sup_{\bm{\alpha} \in \mathbb{S}^{sm-1}} \left | \frac{1}{n} \sum_{i=1}^{n} (\bm{\alpha}'\bm{x}_{iG_{T}})^{2} - \mathrm{E}[(\bm{\alpha}'\bm{x}_{1G_{T}})^{2}] \right |.
\end{equation*} 
Applying Massart's (2000) form of Talagrand's (1996) inequality to the right sum, for all $t > 0$, with probability at least $1-e^{-t}$, we have 
\begin{equation*}
\| \hat{\bm{\Sigma}}_{0T} - \bm{\Sigma}_{T} \| \leq 2\mathrm{E}[\| \hat{\bm{\Sigma}}_{0T} - \bm{\Sigma}_{T} \|] + C K \phi_{\max} (s) \sqrt{tsm/n} + CK^{2} tsm/n,
\end{equation*}
where we have use the fact that $| \bm{\alpha}'\bm{x}_{1G_{T}} | \leq K \sqrt{sm}$ and $\mathrm{E}[(\bm{\alpha}'\bm{x}_{1G_{T}})^{4}] \leq K^{2} sm \mathrm{E}[(\bm{\alpha}'\bm{x}_{1G_{T}})^{2}] \leq K^{2} sm \phi_{\max} (s)^{2}$. We now bound the expectation $\mathrm{E}[\| \hat{\bm{\Sigma}}_{0T} - \bm{\Sigma}_{T} \|]$ by Rudelson's (1999) inequality:
\begin{equation*}
\mathrm{E}[\| \hat{\bm{\Sigma}}_{0T} - \bm{\Sigma}_{T} \|] \leq \max \left \{ CK \phi_{\max}(s) \sqrt{\frac{sm\log (sm)}{n}}, C^{2} K^{2} \frac{sm \log (sm)}{n} \right \}.
\end{equation*}
Thus, for all $t > 0$, with probability at least $1-e^{-t}$, we have 
\begin{equation*}
\| \hat{\bm{\Sigma}}_{0T} - \bm{\Sigma}_{T} \| \leq C K \max \left \{  \phi_{\max}(s) \sqrt{\frac{sm(t \vee \log (sm))}{n}},  \frac{K sm (t \vee \log (sm))}{n} \right \}.
\end{equation*}
Because the number of all subsets $T$ of $\{ 1,\dots, d \}$ such that $|T| = s$ is $\binom{d}{s} \leq (ed/s)^{s}$, the above inequality simultaneously holds for all such $T$ with probability at least $1-\exp \{ s \log (ed/s) - t\}$. Taking $t=2 s\log (ed/s)$ and recalling condition (\ref{growth}), we have
\begin{equation}
\max_{|T| \leq s} \| \hat{\bm{\Sigma}}_{0T} - \bm{\Sigma}_{T} \| = \max_{| T | =s} \| \hat{\bm{\Sigma}}_{0T} - \bm{\Sigma}_{T} \| \lesssim_{p} o(1) (\phi_{\max}(s)  \vee 1)= o(1),
\label{one}
\end{equation}
provided that $\phi_{\max}(s) \lesssim 1$. 

It remains to bound $\| \bar{\bm{x}}_{G_{T}} \|_{E}$. By Corollary 4.5 of \cite{L01},  
for all $t > 0$, with probability at least $1-e^{-t^{2}/2}$, we have 
\begin{equation*}
\| \bar{\bm{x}}_{G_{T}} \|_{E} \leq (1+CKt) \sqrt{sm/n},
\end{equation*}
where we have used the fact that $\mathrm{E}[ \| \bar{\bm{x}}_{G_{T}} \|_{E}] \leq \sqrt{n^{-2} \sum_{i=1}^{n} \mathrm{E}[\| \bm{x}_{iG_{T}} \|_{E}^{2}}]=\sqrt{sm/n}$ and $| \bm{\alpha}'\bm{x}_{1G_{T}} | \leq K\sqrt{sm}$. 
Taking $t=2\sqrt{s\log (ed/s)}$, we have 
\begin{equation}
\max_{|T| \leq s} \| \bar{\bm{x}}_{G_{T}} \|_{E} = \max_{|T|=s} \| \bar{\bm{x}}_{G_{T}} \|_{E} = o_{p}(1).
\label{two}
\end{equation}

Combining (\ref{one}) and (\ref{two}), we have 
\begin{equation*}
\hat{\phi}_{\max}(s)^{2} \leq \phi_{\max}(s)^{2} + \max_{| T | \leq s} \| \hat{\bm{\Sigma}}_{T} - \bm{\Sigma}_{T} \| = \phi_{\max}(s)^{2} + o_{p}(1) \lesssim_{p} 1.
\end{equation*}
\end{proof}

\begin{lemma}
\label{lem5}
Let $\kappa$ denote the $\mathbb{C}$-restricted eigenvalue of $\bm{\Sigma}^{1/2}$: \\
$\kappa := \min_{\bm{\alpha} \in \mathbb{S}^{dm-1} \cap \mathbb{C}} \| \bm{\Sigma}^{1/2} \bm{\alpha} \|_{E}$. 
Assume that $s^{*},m,d$ and $n$ obey the growth condition 
\begin{equation}
\frac{(s^{*})^{2}m\log (d \vee n)}{n} \to 0.
\label{growth2}
\end{equation}
Then, we have $\hat{\kappa} \gtrsim_{p} 1$ provided that $\phi_{\max} (2s^{*}) \lesssim 1$ and $\kappa \gtrsim 1$. 
\end{lemma}

\begin{proof}[Proof of Lemma \ref{lem5}]
We partly use an idea of \cite{Zhou09}, but the overall proof is quite different.
Put $c_{0} = 21$.
For any $\bm{\alpha} \in \mathbb{C}$, we decompose $\bm{\alpha}$ into a set of subvectors $\bm{\alpha}_{G_{T_{0}}}, \bm{\alpha}_{G_{T_{1}}},\dots,\bm{\alpha}_{G_{T_{L}}}$ in such a way that $T_{0}$ corresponds to the $s^{*}$ largest groups in $\bm{\alpha}$ in the Euclidean norm, $T_{1}$ corresponds to the $s^{*}$ largest groups in $\bm{\alpha}_{G_{T_{0}^{c}}}$, and so on. Then, we have 
$T_{0}^{c} = \cup_{l=1}^{L} T_{l}$ where $| T_{l} | = s^{*}$ for $1 \leq l \leq L-1$ and $| T_{L} | \leq s^{*}$. Because for each $l \geq 1$,
\begin{equation*}
\| \bm{\alpha}_{G_{T_{l}}} \|_{E} \leq \sqrt{s^{*}} \max_{j \in T_{l}} \| \bm{\alpha}_{G_{j}} \|_{E} \leq \frac{1}{\sqrt{s^{*}}} \sum_{j \in T_{l-1}} \| \bm{\alpha}_{G_{j}} \|_{E},
\end{equation*}
we have
\begin{align*}
\sum_{l=1}^{L} \| \bm{\alpha}_{G_{T_{l}}} \|_{E} &\leq \frac{1}{\sqrt{s^{*}}} \sum_{l=0}^{L-1} \sum_{j \in T_{l}} \| \bm{\alpha}_{G_{j}} \|_{E} \\
&\leq \frac{1}{\sqrt{s^{*}}} \sum_{j=1}^{d} \| \bm{\alpha}_{G_{j}} \|_{E} \\
&\leq \frac{1}{\sqrt{s^{*}}} (1+c_{0}) \sum_{j \in T^{*}} \| \bm{\alpha}_{G_{j}} \|_{E} \\
&\leq (1+c_{0}) \| \bm{\alpha}_{G_{T^{*}}} \|_{E} \\
&\leq (1+c_{0}) \| \bm{\alpha}_{G_{T_{0}}} \|_{E},
\end{align*}
where we have use the fact that $\sum_{j \in T^{*}} \| \bm{\alpha}_{G_{j}} \|_{E} \leq c_{0} \sum_{j \in (T^{*})^{c}} \| \bm{\alpha}_{G_{j}} \|_{E}$, and 
$\| \bm{\alpha}_{G_{T^{*}}} \|_{E} \leq \| \bm{\alpha}_{G_{T_{0}}} \|_{E}$ by construction. Therefore, we have 
\begin{equation}
\sum_{l=0}^{L} \| \bm{\alpha}_{G_{T_{l}}} \|_{E} \leq (2+c_{0}) \| \bm{\alpha}_{G_{T_{0}}} \|_{E}.
\label{bound}
\end{equation}

In what follows, we identify $\bm{\alpha}_{G_{T_{l}}}$ as the $dm \times 1$ vector $\tilde{\bm{\alpha}}$ such that $\tilde{\bm{\alpha}}_{G_{T_{l}}} = \bm{\alpha}_{G_{T_{l}}}$ and all the other elements of $\tilde{\bm{\alpha}}$ are zero. Under this identification, $\bm{\alpha}$ can be written as $\bm{\alpha} = \sum_{l=0}^{L} \bm{\alpha}_{G_{T_{l}}}$. 
Invoke now that
\begin{align}
&| \bm{\alpha}' (\hat{\bm{\Sigma}} - \bm{\Sigma})\bm{\alpha} | \notag \\
&\leq \sum_{l=0}^{L} \sum_{l'=0}^{L} |  \bm{\alpha}_{G_{T_{l}}}'(\hat{\bm{\Sigma}} - \bm{\Sigma}) \bm{\alpha}_{G_{T_{l'}}} | \notag \\
&= \sum_{l=0}^{L}\| \bm{\alpha}_{G_{T_{l}}} \|_{E} \sum_{l'=0}^{L}\| \bm{\alpha}_{G_{T_{l'}}} \|_{E}   \cdot | \bm{h}_{l}'(\hat{\bm{\Sigma}} - \bm{\Sigma}) \bm{h}_{l'}| \quad (\bm{h}_{l}:=\bm{\alpha}_{G_{T_{l}}}/\| \bm{\alpha}_{G_{T_{l}}} \|_{E}) \notag \\
&\leq \max_{0 \leq l,l' \leq L}  | \bm{h}_{l}'(\hat{\bm{\Sigma}} - \bm{\Sigma}) \bm{h}_{l'}| \cdot \left (\sum_{l=0}^{L}\| \bm{\alpha}_{G_{T_{l}}} \|_{E} \right )^{2} \notag \\
&\leq (2+c_{0})^{2} \| \bm{\alpha}_{G_{T_{0}}} \|_{E}^{2} \max_{0 \leq l,l' \leq L}  | \bm{h}_{l}'(\hat{\bm{\Sigma}} - \bm{\Sigma}) \bm{h}_{l'}|, \label{ineqb1}
\end{align}
where we have used (\ref{bound}).
Take and fix any $l,l'$ such that $0 \leq l, l' \leq L$, and let $T' = T_{l} \cup T_{l'}$. It is not hard to see that 
\begin{align*}
| \bm{h}_{l}'(\hat{\bm{\Sigma}} - \bm{\Sigma}) \bm{h}_{l'}| \leq \| \hat{\bm{\Sigma}}_{T'} - \bm{\Sigma}_{T'} \|. 
\end{align*}
Since $| T' | \leq 2s^{*}$, we see that 
\begin{equation}
\max_{0 \leq l,l' \leq L}  | \bm{h}_{l}'(\hat{\bm{\Sigma}} - \bm{\Sigma}) \bm{h}_{l'}| \leq \max_{| T | \leq 2s^{*}} \| \hat{\bm{\Sigma}}_{T} - \bm{\Sigma}_{T} \|. \label{ineqb2}
\end{equation}

Combining (\ref{ineqb1}) and (\ref{ineqb2}), we have
\begin{equation*}
| \hat{\kappa}^{2} - \kappa^{2} | \leq \max_{\bm{\alpha} \in \mathbb{S}^{dm-1} \cap \mathbb{C}} | \bm{\alpha}'(\hat{\bm{\Sigma}} - \bm{\Sigma}) \bm{\alpha} | \leq (2+c_{0})^{2} \max_{|T| \leq 2s^{*}} \| \hat{\bm{\Sigma}}_{T} - \bm{\Sigma}_{T} \|.
\end{equation*}
By the proof of Lemma \ref{lem3}, if (\ref{growth2}) and $\phi_{\max}(2s^{*}) \lesssim 1$, we have $\max_{|T| \leq 2s^{*}} \| \hat{\bm{\Sigma}}_{T} - \bm{\Sigma}_{T} \| = o_{p}(1)$. Therefore, we have $\hat{\kappa}^{2} \geq \kappa^{2} - o_{p}(1) \gtrsim_{p} 1$ provided that $\kappa \gtrsim 1$. 

\end{proof}


\begin{lemma}
\label{lem6}
If $\| g^{*} \|_{2}^{2} \lesssim s^{*}$, then $\inf_{g \in \tilde{\mathcal{G}}_{m}^{T^{*}}} \| g^{*} - g \|_{2,n}^{2} \lesssim_{p} s^{*} \max \{  m^{-2\nu}, n^{-1} \}$.
\end{lemma}

\begin{proof}[Proof of Lemma \ref{lem6}] 
Let $g^{m}$ denote an element of $\mathcal{G}_{m}^{T^{*}}$ such that $\| g^{*} - g^{m}\|_{2} = \inf_{g \in \mathcal{G}_{m}^{T^{*}}} \| g^{*} - g \|_{2}$. 
$g^{m}$ can be written as $g^{m}(\bm{z}) = \sum_{j \in T^{*}} g_{j}^{m}(z_{j}) \ (\bm{z} = (z_{1},\dots,z_{d})')$ and $g_{j}^{m}(\cdot) = c_{j}^{*} + \sum_{k=1}^{m} \beta_{jk}^{*} \psi_{k}(\cdot)$ for some $c_{j}^{*} \in \mathbb{R}, \beta_{jk}^{*} \in \mathbb{R} \ (1 \leq k \leq m)$ for each $j \in T^{*}$.
Because $\mathrm{E} [g^{*}(\bm{z}_{1})] = 0$, we have $\mathrm{E}[ g^{m}(\bm{z}_{1}) ] = 0$, so that each $c_{j}^{*}$ may be taken such that
$g_{j}^{m}(\cdot) = \sum_{k=1}^{m} \beta_{jk}^{*} (\psi_{k}(\cdot) - \mathrm{E}[\psi_{k}(z_{1j})])$. Take $\tilde{g}^{m} (\bm{z}) := \sum_{j \in T^{*}} \sum_{k=1}^{m} \beta_{jk}^{*} (\psi_{k}(z_{j}) - \bar{\psi}_{jk})$. Invoke now that
\begin{align*}
\| g^{*} - \tilde{g}^{m} \|_{2,n}^{2} &\leq 2 \| g^{*} - g^{m} \|_{2,n}^{2} + 2 \| g^{m} - \tilde{g}^{m} \|_{2,n}^{2} \\
&=2 \| g^{*} - g^{m} \|_{2,n}^{2} + 2 \{ n^{-1} {\textstyle \sum}_{i=1}^{n} g^{m}(\bm{z}_{i}) \}^{2}.
\end{align*} 
By condition (C7)-(c) and Markov's inequality, we have 
\begin{equation*}
\| g^{*} - g^{m} \|_{2,n}^{2} \lesssim_{p} s^{*} m^{-2\nu},
\end{equation*}
while by the fact that $\mathrm{E}[ g^{m}(\bm{z}_{1}) ] = 0$, we have 
\begin{equation*}
 \{ n^{-1} {\textstyle \sum}_{i=1}^{n} g^{m}(\bm{z}_{i}) \}^{2} \lesssim_{p} n^{-1} \| g^{m} \|_{2}^{2} \lesssim n^{-1} (\| g^{*} \|_{2}^{2} + \| g^{*} - g^{m} \|_{2}^{2}) \lesssim s^{*} n^{-1}.
\end{equation*}
Therefore, we have 
\begin{equation*}
\inf_{g \in \tilde{\mathcal{G}}_{m}^{T^{*}}} \| g^{*} - g \|_{2,n}^{2} \leq \| g^{*} - g^{m} \|_{2,n}^{2} \lesssim_{p} s^{*} \max \{  m^{-2\nu}, n^{-1} \}.
\end{equation*}
\end{proof}

\subsection{Proof of Corollary 4.1}

Take $g^{m} = \sum_{j \in T^{*}} g_{j}^{m}$ as in condition (C7)-(c)'. Without loss of generality, we may assume that 
$\mathrm{E}[ g_{j}^{m}(z_{1j}) ] = 0$ for all $j \in T^{*}$, so that each $g_{j}^{m}$ is written as $g_{j}^{m}(\cdot) = \sum_{k=1}^{m} \beta_{jk}^{0} (\psi_{k}(\cdot) - \mathrm{E}[ \psi_{k}(z_{1j})])$ for some $\beta^{0}_{jk} \in \mathbb{R} \ (1 \leq k \leq m)$. Define $\bm{\beta}^{*} \in \mathbb{R}^{dm}$ by
$\beta_{jk}^{*} = \beta^{0}_{jk}$ for $j \in T^{*}$ and $1 \leq k \leq m$, and $\bm{\beta}^{*}_{G_{j}} = \bm{0}$ for $j \in (T^{*})^{c}$. 
Let $\tilde{g}^{m} := \sum_{j=1}^{d} \tilde{g}_{j}^{m}$, where $\tilde{g}_{j}^{m}(\cdot) := \sum_{k=1}^{m} \beta^{*}_{jk}(\psi_{k}(\cdot) - \bar{\psi}_{jk})$.
By the proof of Lemma \ref{lem6}, we have $\| g^{*} - \tilde{g}^{m} \|_{2,n}^{2} \lesssim_{p} s^{*} m^{-2\nu}$. 
Invoke that 
{\footnotesize
\begin{equation*}
\| \hat{\bm{\beta}} - \bm{\beta}^{*} \|_{E}^{2} \leq \hat{\phi}_{\min}(T^{*} \cup \hat{T})^{-2} \| \hat{g} - \tilde{g}^{m} \|_{2,n}^{2} \lesssim_{p} \| \hat{g} - g^{*} \|_{2,n}^{2} + \| g^{*} - \tilde{g}^{m} \|_{2,n}^{2} \lesssim_{p} \frac{s^{*} m \lambda_{1}^{2}}{n^{2}},
\end{equation*}
}
so that
\begin{equation*}
\| {\textstyle \sum}_{j \in T^{*} \backslash \hat{T}^{0}} g_{j}^{m} \|_{2}^{2} \leq \phi_{\max}(s^{*})^{2} \| \bm{\beta}_{G_{T^{*} \backslash \hat{T}^{0}}} \|_{E}^{2} \lesssim_{p} \frac{s^{*} m \lambda_{1}^{2}}{n^{2}}.
\end{equation*}
Therefore, we have
\begin{equation*}
\| {\textstyle \sum}_{j \in T^{*} \backslash \hat{T}}g_{j}^{*} \|_{2}^{2} \leq 2 \| {\textstyle \sum}_{j \in T^{*} \backslash \hat{T}^{0}} g_{j}^{m} \|_{2}^{2} + 2  \| {\textstyle \sum}_{j \in T^{*} \backslash \hat{T}^{0}} (g_{j}^{*} - g_{j}^{m}) \|_{2}^{2} \lesssim_{p} \frac{s^{*} m \lambda_{1}^{2}}{n^{2}}.
\end{equation*}
\qed

\section{On condition (C2)}

Suppose that $\bm{z}_{1},\dots,\bm{z}_{n}$ are given and fixed. 
As argued in Section 2.1, the key property of the error distribution to our rate analysis of Theorems 4.1 and 4.2 is the normal concentration property (around its mean) of a random variable of the form $\sup_{t \in \mathcal{T}} \sum_{i=1}^{n} u_{i} t_{i} $ where $\mathcal{T}$ is a bounded and countable subset of $\mathbb{R}^{n}$, which means that, letting $Z:=\sup_{t \in \mathcal{T}}  \sum_{i=1}^{n} u_{i} t_{i} $, 
\begin{equation}
\mathrm{P} (Z \geq \mathrm{E}[ Z ] +\sigma r) \leq C \exp (-c r^{2}), \ \forall r > 0,
\label{normal}
\end{equation}
where $\sigma^{2} := \sup_{t \in \mathcal{T}} \sum_{i=1}^{n} t_{i}^{2}$, and $c > 0$ and $C > 0$ are fixed constants (see the proofs of Lemmas 6.1 and \ref{lem3}). Condition (C2) gives a primitive sufficient condition for this normal concentration property. See \cite{L01} for an excellent exposition of the concentration of measure phenomenon. 
On the other hand, \cite{MVB09} assumed a uniform subgaussian condition 
\begin{equation}
\mathrm{E} [ \exp ( u^{2}_{1} /L) | \bm{z}_{1} ] \leq M, \ a.s.,
\label{subgauss}
\end{equation} 
for some fixed constants $L > 0$ and $M > 0$, which is weaker than our condition (C2), but their established rate $s^{*} (\log d/n)^{2\nu/(2\nu+1)}$ is suboptimal. \cite{STS11} later showed that the \cite{MVB09} estimator achieves the minimax rate $s^{*} \delta^{2}$ under a somewhat strong assumption that the error term is uniformly bounded.\footnote{\cite{KY10} and \cite{RWY10} dealt with a different estimator and established the rate $s^{*} \delta^{2}$ under different settings. \cite{KY10} assumed that the error term is uniformly bounded, and \cite{RWY10} assumed that the error term is normal independent of explanatory variables.} It is thus of some interest how our rate analysis changes if our condition (C2) is replaced by weaker (\ref{subgauss}). 

To the best of the author's knowledge, it is not known whether the uniform subgaussian condition alone ensures the normal concentration property (\ref{normal}). However, 
by Theorem \ref{thmC1} ahead, under the uniform subgaussian condition, a slightly weaker inequality 
\begin{equation}
\mathrm{P} (Z \geq \mathrm{E}[ Z ] +\sigma r) \leq C \exp (-c r^{2}/\log n), \ \forall r \geq 4, \label{dev-subgauss}
\end{equation}
holds. A careful inspection of the proofs leads to that if condition (C2) is replaced by (\ref{subgauss}), under some modifications to the conditions, the rate of convergence of the second step estimator will be
\begin{equation*}
\max \left \{ s^{*} n^{-2\nu/(2\nu+1)}, |\hat{T} \backslash T^{*}| \tilde{\delta}^{2}, \| {\textstyle \sum}_{j \in T^{*} \backslash \hat{T}} g_{j}^{*} \|_{2}^{2} \right \},
\end{equation*}
in the canonical case, where 
\begin{equation*}
\tilde{\delta} := \max \left \{ n^{-\nu/(2\nu+1)}, \sqrt{\frac{(\log n)(\log d)}{n}} \right \},
\end{equation*}
and rate of convergence of the group Lasso estimator will be $s^{*} \tilde{\delta}^{2}$. So the second step estimator at least  achieves the rate $s^{*} \tilde{\delta}^{2}$. The only difference is the appearance of the additional $\log n$ term, and as long as $\log d/(n \log n) \to 0$, $s^{*} \tilde{\delta}^{2}$ is faster than the rate $s^{*} (\log d/n)^{2\nu/(2\nu+1)}$. It is also expected that, under the uniform subgaussian condition (\ref{subgauss}), the \cite{MVB09} estimator  has the same rate of convergence as $s^{*} \tilde{\delta}^{2}$. 

\subsection{Proof of (\ref{dev-subgauss})}

Recall the $\psi_{\alpha}$-norm:
\begin{equation*}
\| X \|_{\psi_{\alpha}} = \inf \{ s > 0 : \mathrm{E}[ \exp (X^{\alpha}/s^{\alpha}) ] \leq 2 \}, \ \alpha > 0.
\end{equation*}

Let $\mathcal{T}$ be a bounded and countable subset of $\mathbb{R}^{n}$. 
\begin{theorem}
\label{thmC1}
Let $\epsilon_{1},\dots,\epsilon_{n}$ be independent random variables such that $\max_{1 \leq i \leq n} \| \epsilon_{i} \|_{\psi_{2}} \leq C_{\psi}$ for some constant $C_{\psi}$. Put $Z := \sup_{t \in \mathcal{T}} \sum_{i=1}^{n} \epsilon_{i} t_{i}$. Then, for all $r  \geq 4$, we have 
\begin{equation*}
\mathrm{P} \{ Z   \geq \mathrm{E}[Z] + r\sigma C_{\psi} \} \leq C \exp (- c r^{2}/ \log n ), 
\end{equation*}
where $\sigma := \sqrt{\sup_{t \in \mathcal{T}} \sum_{i=1}^{n} t_{i}^{2}}$, and $c > 0$ and $C>0$ are universal constants. 
\end{theorem}

Theorem \ref{thmC1} is essentially proved in \cite{MT08} (but the exponential term is slightly worse in \cite{MT08}: $\exp (-cr^{2}/\log^{2}n)$). For the sake of completeness, we provide a proof of the theorem. 
The proof of this theorem uses some properties of the $\psi_{1}$-norm.

\begin{lemma}\citep[][Theorem 6.21]{LT91}
\label{lemC1}
Let $\xi_{1},\dots,\xi_{n}$ be independent centered random variables. Then,
\begin{equation*}
\| \sum_{i=1}^{n} \xi_{i} \|_{\psi_{1}} \leq C \left ( \mathrm{E}[ | \sum_{i=1}^{n} \xi_{i} |] + \| \max_{1 \leq i \leq n} | \xi_{i} | \|_{\psi_{1}} \right ),
\end{equation*}
where $C$ is a universal constant.
\end{lemma}

For the evaluation of the term $ \| \max_{1 \leq i \leq n} | \xi_{i} | \|_{\psi_{1}}$, we use the next lemma.

\begin{lemma}\citep[][Lemma 2.2.2]{VW96}
\label{lemC2}
Let $\xi_{1},\dots,\xi_{n}$ be any random variables. Then, 
\begin{equation*}
\| \max_{1 \leq i \leq n} | \xi_{i} | \|_{\psi_{1}} \leq C (\log n) \max_{1 \leq i \leq n} \| \xi_{i} \|_{\psi_{1}},
\end{equation*}
where $C$ is a universal constant. 
\end{lemma}

\begin{proof}[Proof of Theorem \ref{thmC1}]

In this proof, $c$ and $C$ denote some universal constants. Their values may change from line to line. 
Let $\epsilon_{i}^{-} := \epsilon_{i} I( | \epsilon_{i} | \leq L) $ and $\epsilon_{i}^{+} := \epsilon_{i} I ( | \epsilon_{i} | > L)$. The constant $L > 0$ is defined later. Define $\check{\mathcal{T}} := \mathcal{T} \cup \{ -t : t \in \mathcal{T} \}$. Let $Z^{-} := \sup_{t \in \mathcal{T}} \sum_{i=1}^{n} \epsilon_{i}^{-} t_{i}$ and $\check{Z}^{+} := \sup_{t \in \tilde{\mathcal{T}}} \sum_{i=1}^{n} \epsilon_{i}^{+} t_{i}$. Clearly, $Z \leq Z^{-} + \check{Z}^{+}$, and by using that $\sum_{i=1}^{n} \epsilon_{i}^{-} t_{i} = \sum_{i=1}^{n} (\epsilon_{i}- \epsilon^{+}) t_{i}  = \sum_{i=1}^{n} \epsilon_{i} t_{i} + \sum_{i=1}^{n} \epsilon^{+}_{i} (-t_{i})$, we have $\mathrm{E}[ Z^{-} ] \leq \mathrm{E}[ Z ] + \mathrm{E}[ \check{Z}^{+} ]$, so that $\mathrm{E}[ Z ] \geq \mathrm{E}[ Z^{-} ] - \mathrm{E}[ \check{Z}^{+} ]$. Observe that 
\begin{align*}
&\mathrm{P} \{ Z   \geq \mathrm{E}[Z] + r\sigma C_{\psi} \}  \leq \mathrm{P} \{ Z^{-} + \check{Z}^{+} \geq \mathrm{E}[Z^{-}] - \mathrm{E}[\check{Z}^{+}] + r\sigma C_{\psi} \} \\
&\qquad \leq \mathrm{P} \{ Z^{-} \geq  \mathrm{E}[Z^{-}] + r\sigma C_{\psi}/2 \} + \mathrm{P} \{ \check{Z}^{+} + \mathrm{E}[\check{Z}^{+}] \geq r\sigma C_{\psi}/2 \},
\end{align*}
Because $| \epsilon^{-}_{i} | \leq 2 L$, by Corollary 4.8 of \cite{L01}, we have 
\begin{equation*}
\mathrm{P}\{ Z^{-} \geq \mathrm{E}[Z^{-}] + r\sigma C_{\psi}/2 \} \leq C \exp ( -c r^{2}C_{\psi}^{2} /L^{2} ), \ \forall r > 0.
\end{equation*}

On the other hand, it is standard to see that $\| (\epsilon_{i}^{+})^{2} \|_{\psi_{1}} = \| \epsilon_{i}^{+} \|_{\psi_{2}}^{2} \leq  C^{2}_{\psi}$ and 
\begin{align*}
\mathrm{E}[\sum_{i=1}^{n} (\epsilon_{i}^{+})^{2}] &\leq n \max_{1 \leq i \leq n} \mathrm{E}[ \epsilon_{i}^{2} I(| \epsilon_{i} | > L) ]] \\
&\leq n \max_{1 \leq i \leq n} \mathrm{E}[\epsilon_{i}^{4}]^{1/2} \mathrm{P}(| \epsilon_{i}| > L)^{1/2} \\
&\leq n C C^{2}_{\psi} \exp ( - c L^{2}/C_{\psi}^{2}).
\end{align*}
Thus, by Lemmas \ref{lemC1} and \ref{lemC2}, 
\begin{align*}
\| \sum_{i=1}^{n} (\epsilon_{i}^{+})^{2} \|_{\psi_{1}} &\leq \| \sum_{i=1}^{n} \{ (\epsilon_{i}^{+})^{2}-\mathrm{E}[(\epsilon_{i}^{+})^{2}] \} \|_{\psi_{1}} + \mathrm{E}[\sum_{i=1}^{n} (\epsilon_{i}^{+})^{2}]  \\
&\leq CC_{\psi}^{2} \{ n \exp ( - c L^{2}/C_{\psi}^{2}) +  \log n \}.
\end{align*}
Take $L=C C_{\psi} \sqrt{\log n}$ such that $\mathrm{E}[\sum_{i=1}^{n} (\epsilon_{i}^{+})^{2}]  \leq C_{\psi}^{2}$ and $\| \sum_{i=1}^{n} (\epsilon_{i}^{+})^{2} \|_{\psi_{1}} \leq C C_{\psi}^{2} \log n$. Because 
$\check{Z}^{+} \leq \sigma \{ \sum_{i=1}^{n} (\epsilon_{i}^{+})^{2} \}^{1/2}$, we have 
\begin{equation*}
\| \check{Z}^{+} \|_{\psi_{2}} \leq C \sigma C_{\psi} \sqrt{\log n}, \ \mathrm{E}[ \check{Z}^{+} ] \leq \sigma C_{\psi}.
\end{equation*}
Therefore, for all $r \geq 4$, we have 
\begin{align*}
\mathrm{P} \{ \check{Z}^{+} + \mathrm{E}[\check{Z}^{+}] \geq r\sigma C_{\psi}/2 \} &\leq \mathrm{P} \{ \check{Z}^{+} \geq r\sigma C_{\psi}/4 \} \\
&\leq C \exp (- c r^{2}/\log n). 
\end{align*}
This completes the proof. 
\end{proof}

\end{document}